\documentclass[runningheads]{llncs}
\usepackage{graphicx}
\usepackage[colorlinks=true]{hyperref}
\usepackage[noadjust]{cite}
\usepackage{wrapfig}
\usepackage[font=footnotesize]{caption}
\usepackage{amsmath}
\usepackage{amssymb}
\usepackage{amsfonts}
\usepackage{algorithm}
\usepackage{algpseudocode}
\usepackage[pdftex,dvipsnames]{xcolor}
\usepackage{siunitx}
\usepackage{multirow}
\usepackage{xfrac}
\usepackage{url} 
\usepackage{booktabs}
\usepackage{microtype}
\usepackage{nicefrac} 
\usepackage{array}
\usepackage{arydshln}
\numberwithin{equation}{section}
\begin{document}
\title{Scaling Up Quasi-Newton Algorithms: Communication Efficient Distributed SR1}
\titlerunning{Communication Efficient Distributed SR1}
% If the paper title is too long for the running head, you can set
% an abbreviated paper title here
%
% \author{First Author\inst{1}\orcidID{0000-1111-2222-3333} \and
% Second Author\inst{2,3}\orcidID{1111-2222-3333-4444} \and
% Third Author\inst{3}\orcidID{2222--3333-4444-5555}}

\author{Majid Jahani \and Mohammadreza Nazari \and Sergey Rusakov \and Albert S. Berahas \and Martin Tak\'a\v c}
\authorrunning{Majid Jahani et al.}
% First names are abbreviated in the running head.
% If there are more than two authors, 'et al.' is used.
%
% \institute{Princeton University, Princeton NJ 08544, USA \and
% Springer Heidelberg, Tiergartenstr. 17, 69121 Heidelberg, Germany
% \email{lncs@springer.com}\\
% \url{http://www.springer.com/gp/computer-science/lncs} \and
% ABC Institute, Rupert-Karls-University Heidelberg, Heidelberg, Germany\\
% \email{\{abc,lncs\}@uni-heidelberg.de}}
\institute{Lehigh University, Bethlehem, PA, 18015, USA}
\maketitle              % typeset the header of the contribution

\begin{abstract}
In this paper, we present a scalable distributed implementation of the Sampled Limited-memory Symmetric Rank-1 (S-LSR1) algorithm. First, we show that a naive distributed implementation of S-LSR1 requires multiple rounds of expensive communications at every iteration and thus is inefficient. We then propose DS-LSR1, a communication-efficient variant % of the S-LSR1 method, 
that: $(i)$ drastically reduces the amount of data communicated at every iteration, $(ii)$ has favorable work-load balancing across nodes, and $(iii)$ is matrix-free and inverse-free. The proposed method scales well in terms of both the dimension of the problem and the number of data points. Finally, we illustrate the empirical performance of DS-LSR1 on a standard neural network training task.

\end{abstract}
%
%
%
%\vspace{-20pt}

\section{Introduction}
In the last decades, significant efforts have been devoted to the development of optimization algorithms for machine learning. Currently, due to its fast learning properties, low per-iteration cost, and ease of implementation, the stochastic gradient (SG) method \cite{robbins1951stochastic,bottou2004large}, and its adaptive \cite{duchi2011adaptive,kingma2014adam}, variance-reduced \cite{johnson2013accelerating,schmidt2017minimizing,defazio2014saga} and distributed \cite{recht2011hogwild,tsitsiklis1986distributed,zinkevich2010parallelized,dean2012large} variants are the preferred optimization methods for large-scale machine learning applications. Nevertheless, these methods have several drawbacks; they are highly sensitive to the choice of hyper-parameters and are cumbersome to tune, and they suffer from ill-conditioning \cite{berahas2017investigation,xu2017second,bottou2018optimization}. More importantly, these methods offer a limited amount of benefit in distributed computing environments since they are usually implemented with small mini-batches, and thus spend more time communicating instead of performing ``actual'' computations.  This shortcoming can be remedied to some extent by increasing the batch sizes, however, there is a point after which the increase in computation is not offset by the faster convergence \cite{takac2013mini}.%More importantly, these methods offer a limited amount of benefit in distributed computing environments. Since these methods are usually implemented with small mini-batches, they spend more time communicating instead of performing ``actual'' computations. This shortcoming can be remedied to some extent by increasing the batch sizes, however, there is a point after which the increase in computation is not offset by the faster convergence \cite{takac2013mini}.

Recently, there has been an increased interest in (stochastic) second-order and quasi-Newton methods by the machine learning community; see e.g., \cite{byrd2011use,martens2010deep,bollapragada2016exact,Roosta-Khorasani2018,xu2017newton,schraudolph2007stochastic,byrd2016stochastic,curtis2016self,berahas2016multi,berahas2017robust,jahani2018efficient,keskar2016adaqn}. These methods judiciously incorporate curvature information, and thus mitigate some of the issues that plague first-order methods. Another benefit of these methods is that they are usually implemented with larger batches, and thus better balance the communication and computation costs. Of course, this does not come for free; (stochastic) second-order and quasi-Newton methods are more memory intensive and more expensive (per iteration) than first-order methods. This naturally calls for distributed implementations.

\begin{wrapfigure}{r}{0.35\textwidth}
\vspace{-5 pt}
  \centering
    \includegraphics[width=0.34\textwidth]{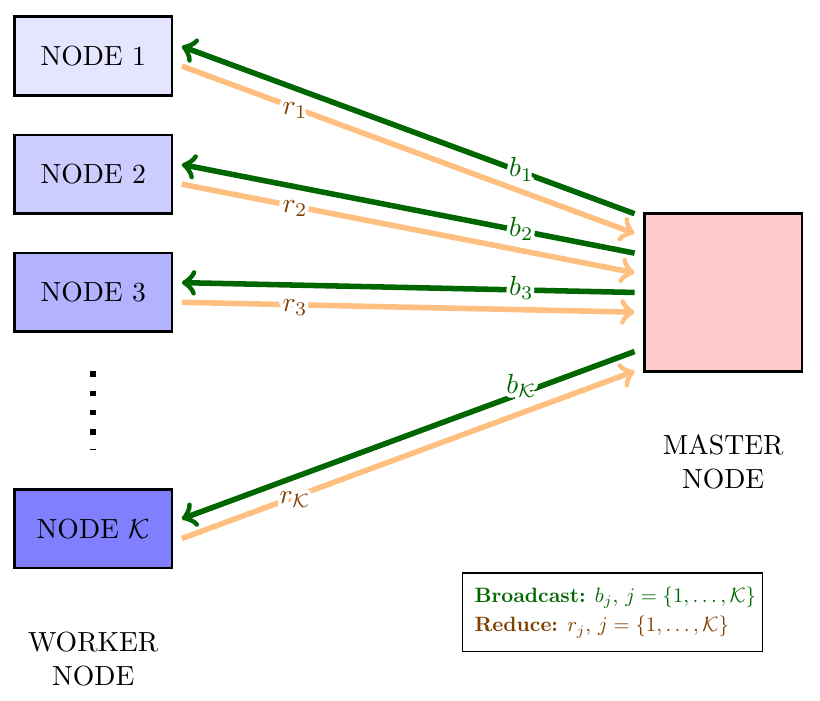}
    \vspace{-8 pt}
  \caption{ Distributed Computing Schematic.}
  \label{fig:dist_frame}
  \vspace{-24 pt}
\end{wrapfigure}
In this paper, we propose an efficient distributed variant of the Sampled Limited-memory Symmetric Rank-1 (S-LSR1) method \cite{berahas2019sqn}---DS-LSR1---that operates in the master-worker framework (Figure \ref{fig:dist_frame}). Each worker node has a portion of the dataset, and performs local computations using solely that information and information received from the master node. The proposed method is matrix-free (Hessian approximation never explicitly constructed) and inverse-free (no matrix inversion). To this end, we leverage the compact form of the SR1 Hessian approximations \cite{byrd1994},
and utilize sketching techniques \cite{woodruff2014sketching} to approximate several required quantities. We show that, contrary to a naive distributed implementation of S-LSR1, the method is communication-efficient and has favorable work-load balancing across nodes. %Specifically, the naive implementation requires communicating $\mathcal{O}(md)$ quantities, whereas our approach only requires communicating $\mathcal{O}(m^2)$ quantities, where $d$ is the dimension of the problem and $m$ is the LSR1 memory
Specifically, the naive implementation requires $\mathcal{O}(md)$ communication, whereas our approach only requires $\mathcal{O}(m^2)$ communication, where $d$ is the dimension of the problem, $m$ is the LSR1 memory and $m\ll d$\footnote{\scriptsize{Note, these costs are on top of the communications that are common to both approaches.}}. Furthermore, in our approach the heavy computations are done by the worker nodes and the master node performs only simple aggregations, whereas in the naive approach the computationally intensive operations, e.g., Hessian-vector products, are computed locally by the master node. Finally, we show empirically that DS-LSR1 has good strong and weak scaling properties, and illustrate the performance of the method on a standard neural network training task. 

\paragraph{Problem Formulation and Notation}
We focus on machine learning empirical risk minimization problems that can be expressed as:
\begin{equation}   \label{eq:opt_prob}
    \min_{w \in \mathbb{R}^d} F(w) := \frac{1}{n} \sum_{i=1}^n f(w;x^i,y^i) = \frac{1}{n} \sum_{i=1}^n f_i(w),
\end{equation}
where $f: \mathbb{R}^d \rightarrow \mathbb{R}$ is the composition of a prediction function
(parametrized by $w$) and a loss function, and $(x^i,y^i)_{i=1}^n$ denote the training examples (samples). Specifically, we focus on deep neural network training tasks where the function $F$ is nonconvex, and the dimension $d$ and number of samples $n$ are large. 

The paper is organized as follows. We conclude this section with a discussion of related work. We describe the classical (L)SR1  and sampled LSR1 (S-LSR1) methods in Section \ref{sec:slsr1}. In Section \ref{sec:ds-lsr1}, we present DS-LSR1, our proposed distributed variant of the sampled LSR1 method. We illustrate the scaling properties of DS-LSR1 and the empirical performance of the method on deep learning tasks in Section \ref{sec:num_res}. Finally, in Section \ref{sec:fin_rem} we provide some final remarks. 

\paragraph{Related Work} The Symmetric Rank-1 (SR1) method \cite{conn1991convergence,khalfan1993theoretical} and its limited-memory variant (LSR1) \cite{lu1996study} are quasi-Newton methods that have gained significant attention by the machine learning community in recent years \cite{berahas2019sqn,erway2018trust}. These methods incorporate curvature (second-order) information using only gradient (first-order) information. Contrary to  arguably the most popular quasi-Newton method, (L)BFGS \cite{nocedal_book,liu1989limited}, the (L)SR1 method does not enforce that the Hessian approximations are positive definite, and is usually implemented with a trust-region \cite{nocedal_book}. This has several benefits: $(1)$ the method is able to exploit negative curvature, and $(2)$ the method is able to efficiently escape saddle points.

There has been a significant volume of research on distributed algorithms for machine learning; specifically, distributed gradient methods \cite{zinkevich2010parallelized,bottou2010large,recht2011hogwild,tsitsiklis1986distributed,chu2007map}, distributed Newton methods \cite{shamir2014communication,jahani2018efficient,zhang2015communication} and distributed quasi-Newton methods \cite{chen2014large,dean2012large,agarwal2014reliable}. %General distributed optimization methods close to our work are the approaches based on parallel gradient computation followed by a centralized algorithm \cite{chen2014large,gopal2013distributed}.
Possibly the closest work to ours is VF-BFGS \cite{chen2014large}, in which the authors propose a vector-free implementation of the classical LBFGS method. We leverage several of the techniques proposed in \cite{chen2014large}, however, what differentiates our work is that we focus on the S-LSR1 method. Developing an efficient distributed implementation of the S-LSR1 method is not as straight-forward as LBFGS for several reasons: $(1)$ the construction and acceptance of the curvature pairs, $(2)$ the trust-region subproblem, and $(3)$ the step acceptance procedure. 

\section{Sampled limited-memory SR1 (S-LSR1)}
\label{sec:slsr1}

In this section, we review the sampled LSR1 (S-LSR1) method \cite{berahas2019sqn}, and discuss the components that can be distributed. We begin by describing the classical (L)SR1 method as this will set the stage for the presentation of the S-LSR1 method. At the $k$th iteration, the SR1 method computes a new iterate via
\begin{align*}  
    w_{k+1} = w_k + p_k,
\end{align*}
where $p_k$ is the minimizer of the following subproblem
\begin{align}   
    &\textstyle{\min_{\| p \| \leq \Delta_k}}  \; m_k(p) = F(w_k) + \nabla F(w_k)^Tp + \tfrac{1}{2} p^T B_k p , \label{eq:tr_obj}
\end{align}
$\Delta_k$ is the trust region radius, $B_k$ is the SR1 Hessian approximation 
\begin{align}   \label{eq:sr1_hess}
    B_{k+1} = B_k +  \tfrac{(y_k - B_ks_k)(y_k - B_ks_k)^T}{(y_k - B_ks_k)^Ts_k},
\end{align}
and  $(s_k,y_k) = (w_{k} - w_{k-1}, \nabla F(w_{k}) - \nabla F(w_{k-1}))$ are the curvature pairs. In the limited memory version, the matrix $B_k$ is defined as the result of applying $m$ SR1 updates to a multiple of the identity matrix using the set of $m$ most recent curvature pairs $\{s_i, y_i\}_{i=k-m}^{k-1}$ kept in storage.

The main idea of the S-LSR1 method is to use the SR1 updating formula, but to construct the Hessian approximations using sampled curvature pairs instead of pairs that are constructed as the optimization progresses. At every iteration, $m$ curvature pairs are constructed via random sampling around the current iterate; see Algorithm \ref{alg:calSY1}. The S-LSR1 method is outlined in Algorithm \ref{alg:slsr1}. The components of the algorithms that can be distributed are highlighted in \textcolor{magenta}{magenta}.

Several components of the above algorithms can be distributed. Before we present the distributed implementations of the S-LSR1 method, we discuss several key elements of the method:  $(1)$ Hessian-vector products; $(2)$ curvature pair construction; $(3)$ curvature pair acceptance; $(4)$ search direction computation; $(5)$ step acceptance procedure; and $(6)$ initial Hessian approximations. 

For the remainder of the paper, let $S_k=[s_{k,1},s_{k,2},\dots, s_{k,m}] \in \mathbb{R}^{d\times m}$ and $Y_k=[y_{k,1},y_{k,2},\dots, y_{k,m}] \in \mathbb{R}^{d\times m}$ denote the curvature pairs constructed at the $k$th iteration, $S_{k}^i \in \mathbb{R}^{d\times m}$ and $Y_{k}^i \in \mathbb{R}^{d\times m}$ denote the curvature pairs constructed at the $k$th iteration by the $i$th node, and $B_k^{(0)} = \gamma_k I \in \mathbb{R}^{d\times d}$, $\gamma_k\geq0$, denote the initial Hessian approximation at the $k$th iteration. 

\begin{minipage}{0.52\textwidth}
\begin{center}
\begin{algorithm}[H]
 {\small 
\caption{Sampled LSR1 (S-LSR1)}
  \label{alg:slsr1}
 {\bf Input:} $w_{0}$ (initial iterate), $\Delta_0$ (initial trust region radius).%, $m$ (memory), $r$ (sampling radius).

  \begin{algorithmic}[1]
  \For {$k=0,1,2,...$}
    \State Compute \textcolor{magenta}{$F(w_k)$} and \textcolor{magenta}{$ \nabla F(w_k)$}
    \State Compute $(S_k,Y_k)$ (Algorithm \ref{alg:calSY1})
%    \State Compute $B_{k+1}$ via \eqref{eq:sr1_hess}
    \State Compute $p_k$ (solve subproblem \eqref{eq:tr_obj})
    \State Compute $\rho_k = \frac{F(w_k) - \textcolor{magenta}{F(w_k + p_k)}}{\textcolor{magenta}{m_k(0) - m_k(p_k)}}$
    \State \textbf{if } $\rho_k \geq \eta_1$ \textbf{then } Set $w_{k+1} = w_k + p_k$
    \State \textbf{else } Set $w_{k+1} = w_k$
    % \If {$\rho_k \geq \eta_1$}
    %     \State Set $w_{k+1} = w_k + p_k$
    % \Else
    %     \State Set $w_{k+1} = w_k$
    % \EndIf
    \State $\Delta_{k+1} = \texttt{adjustTR}(\Delta_{k},\rho_k)$ \cite[Appendix B.3]{berahas2019sqn}
\EndFor
  \end{algorithmic}
  }
\end{algorithm}
\end{center}
\end{minipage}
\begin{minipage}{0.46\textwidth}
\begin{center}
\begin{algorithm}[H]
{ 
\small
\caption{Construct new $(S_k,Y_k)$ curvature pairs}
  \label{alg:calSY1}
 {\bf Input:} $w_k$ (current iterate), $m$ (memory), $r$ (sampling radius), $S_k = [\;]$, $Y_k =[\;]$ (curvature pair containers). 

  \begin{algorithmic}[1]
  \For {$i=1,2,...,m$}
        \State Sample a random direction $\sigma_i$
        \State Sample point $\bar{w}_i = w_k + r\sigma_i $
        \State Set $s_i = w_k - \bar{w}_i$ and
        
        \textcolor{magenta}{$ \quad \; y_i = \nabla^2 F(w_k)s_i$}
        \State Set $S_k = [S_k \; s_i]$ and $ Y_k = [Y_k \; y_i]$
    \EndFor
    %\State Check condition \eqref{eq:cond_slsr1} and return $S_k$, $Y_k$
  \end{algorithmic}
  {\bf Output:} $S$, $Y$
  }
\end{algorithm}
\end{center}
\end{minipage}

\paragraph{Hessian-vector products} Several components of the algorithms above require the calculation of Hessian vector products of the form $B_k v$. In the large-scale setting, it is not memory-efficient, or even possible for some applications, to explicitly compute and store the $d\times d$ Hessian approximation matrix $B_{k}$. Instead, one can exploit the compact representation of the SR1 matrices \cite{byrd1994} and compute:% Hessian vector products efficiently as follows:
\begin{gather}   %\label{eq:compactFormVec}
     B_{k+1}v = B_k^{(0)} v+ (Y_k- B_k^{(0)}S_k)( \underbrace{D_k+L_k +L_k^T -S_k^TB_k^{(0)}S_k}_{M_k})^{-1}(Y_k-B_k^{(0)}S_k)^Tv, \nonumber\\
     D_k = {diag}[s_{k,1}^Ty_{k,1}, \dots ,s_{k,m}^T y_{k,m}], \quad (L_k)_{j,l} = \begin{cases}
  s_{k,j-1}^T y_{k,l-1} & \text{if }j>l, \\
  0 & \text{otherwise} \label{eq:hessFreeD}.
\end{cases}
\end{gather}
% where 
% \begin{align}   \label{eq:hessFreeD}
%  D_k = {diag}[s_{k,1}^Ty_{k,1}, \dots ,s_{k,m}^T y_{k,m}], \quad (L_k)_{j,l} = \begin{cases}
%   s_{k,j-1}^T y_{k,l-1} & \text{if }j>l, \\
%   0 & \text{otherwise}.
% \end{cases}
% \end{align}
Computing $B_{k+1}v$ via \eqref{eq:hessFreeD} is both memory and computationally efficient; the complexity of computing $B_{k+1}v$ is $\mathcal{O}(m^2d)$ \cite{byrd1994}. 

\paragraph{Curvature pair construction} For ease of exposition, we presented the curvature pair construction routine (Algorithm \ref{alg:calSY1}) as a sequential process. However, this need not be the case; all pairs can be constructed simultaneously. First, generate a random matrix $S_k \in \mathbb{R}^{d \times m}$, and then compute $Y_k = \nabla^2 F(w_k)S_k \in \mathbb{R}^{d \times m}$. We discuss a distributed implementation of this routine in the following sections. 

\paragraph{Curvature pair acceptance} In order for the S-LSR1 Hessian update \eqref{eq:sr1_hess} to be well defined, and for numerical stability, we require certain conditions on the curvature pairs employed; see \cite[Chapter 6]{nocedal_book}. Namely, for a given $\eta>0$, we impose that the Hessian approximation $B_{k+1}$ is only updated using the curvature pairs that satisfy the following condition:
\begin{align}   \label{eq:cond_slsr1}
    | s_{k,j}^T(y_{k,i}-B_k^{(j-1)}s_{k,j}) | \geq \eta \| s_{k,j}\| \| y_{k,i}-B_k^{(j-1)}s_{k,j} \|,
\end{align}
for $j=1,\dots,m$, where $B_k^{(0)}$ is the initial Hessian approximation and $B_k^{(j-1)}$, for $j=2,\dots,m$, is the Hessian approximation constructed using only curvature pairs $\{s_l,y_l \}$, for $l<j$, that satisfy \eqref{eq:cond_slsr1}. Note, $B_{k+1} = B_{k}^{(m)}$. Thus, potentially, not all curvature pairs returned by Algorithm \ref{alg:calSY1} are used to update the S-LSR1 Hessian approximation. Checking this condition is not trivial and requires $m$ Hessian vector products. In \cite[Appendix B.5]{berahas2019sqn}, the authors propose a recursive memory-efficient mechanism to check and retain only the pairs that satisfy \eqref{eq:cond_slsr1}. 

\paragraph{Search direction computation} The search direction $p_k$ is computed by solving subproblem \eqref{eq:tr_obj} using CG-Steihaug; see \cite[Chapter 7]{nocedal_book}%; see Appendix \ref{sec:CG_serial} Algorithm \ref{alg:CG-Steihaug(Serial)}
. This procedure requires the computation of Hessian vectors products of the form \eqref{eq:hessFreeD}. 

\paragraph{Step acceptance procedure} In order to determine if a step is successful (Line 6, Algorithm \ref{alg:slsr1}) one has to compute the function value at the trial iterate and the predicted model reduction. This entails a function evaluation and a Hessian vector product. The acceptance ratio $\rho_k$ determines if a step is successful, after which the trust region radius has to be adjusted accordingly. For brevity we omit the details from the paper and refer the interested reader to \cite[Appendix B.3]{berahas2019sqn}. 

\paragraph{Initial Hessian approximations $B_k^{(0)}$} 
% In practice, it is not clear how the initial Hessian approximation should be chosen. We argue, that in the context of the S-LSR1 method, a good choice is $B_k^{(0)} = 0$. In Figure \ref{SmalllvdNet} we show the eigenvalues of the true Hessian and the eigenvalues of the S-LSR1 matrices for different values of $\gamma_k$ for a toy problem \cite{berahas2019sqn}. As is clear, the eigenvalues of the S-LSR1 matrices with $\gamma_k=0$ better match the eigenvalues of the true Hessian. Similar results were observed for other datasets; see Appendix \ref{sec:init_Hess}. Moreover, by setting $\gamma_k= 0$, the rank of the approximation is at most $m$ and thus the CG algorithm will terminate in at most $m$ iterations, whereas the CG algorithm may require as many as $d\gg m$ iterations when $\gamma_k \neq 0$. Another reason for making this choice is that it removes a hyper-parameter. Henceforth, we assume that $B_k^{(0)} = 0$, however, we note that our method can be extended to $B_k^{(0)} \neq 0$.
In practice, it is not clear how to choose the initial Hessian approximation. We argue, that in the context of S-LSR1, a good choice is $B_k^{(0)} = 0$. In Figure \ref{SmalllvdNet} we show the eigenvalues of the true Hessian and the eigenvalues of the S-LSR1 matrices for different values of $\gamma_k$ ($B_k^{(0)} = \gamma_kI$) for a toy problem. As is clear, the eigenvalues of the S-LSR1 matrices with $\gamma_k=0$ better match the eigenvalues of the true Hessian. Moreover, by setting $\gamma_k= 0$, the rank of the approximation is at most $m$ and thus the CG algorithm (used to compute the search direction) terminates in at most $m$ iterations, whereas the CG algorithm may require as many as $d\gg m$ iterations when $\gamma_k \neq 0$. Finally, $B_k^{(0)} = 0$ removes a hyper-parameter. Henceforth, we assume that $B_k^{(0)} = 0$, however, we note that our method can be extended to $B_k^{(0)} \neq 0$. 
\begin{figure}[]
	\centering
	\includegraphics[trim=0.6cm 0.5cm 1.0cm 0.7cm,width=0.24\textwidth]{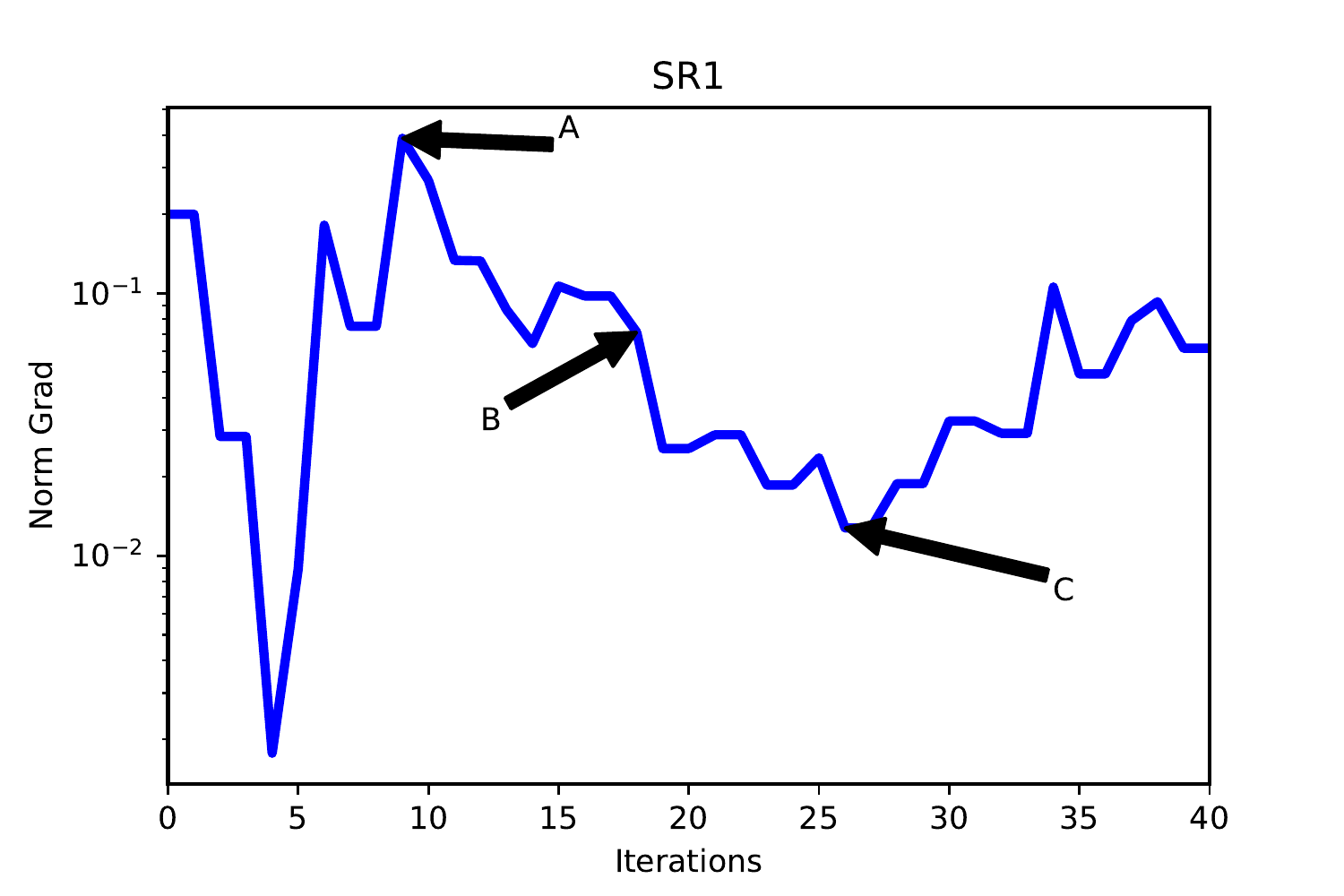}
 	\includegraphics[trim=0.6cm 0.7cm 1.0cm 0.7cm,width=0.24\textwidth]{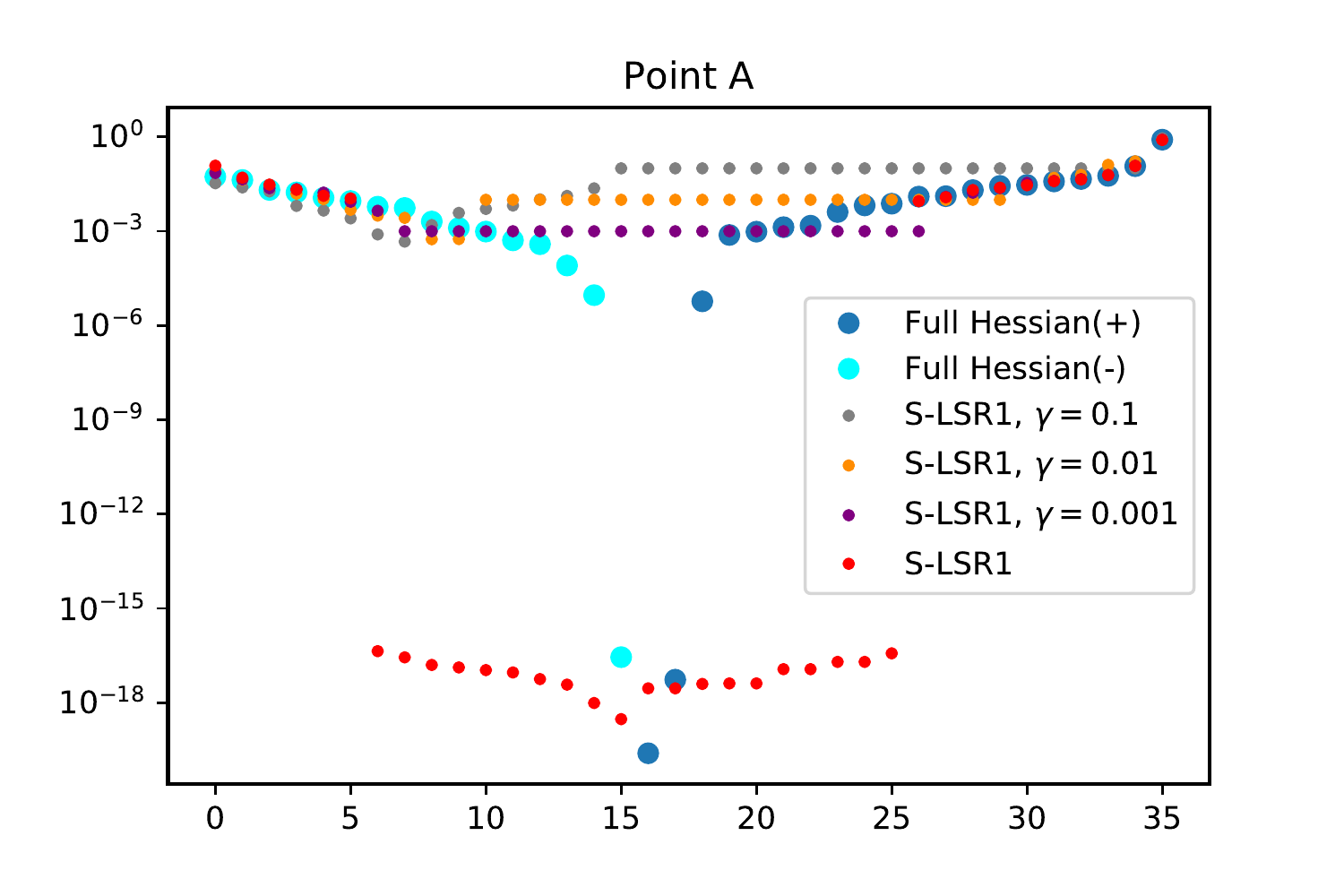}
	\includegraphics[trim=0.6cm 0.7cm 1.0cm 0.7cm,width=0.24\textwidth]{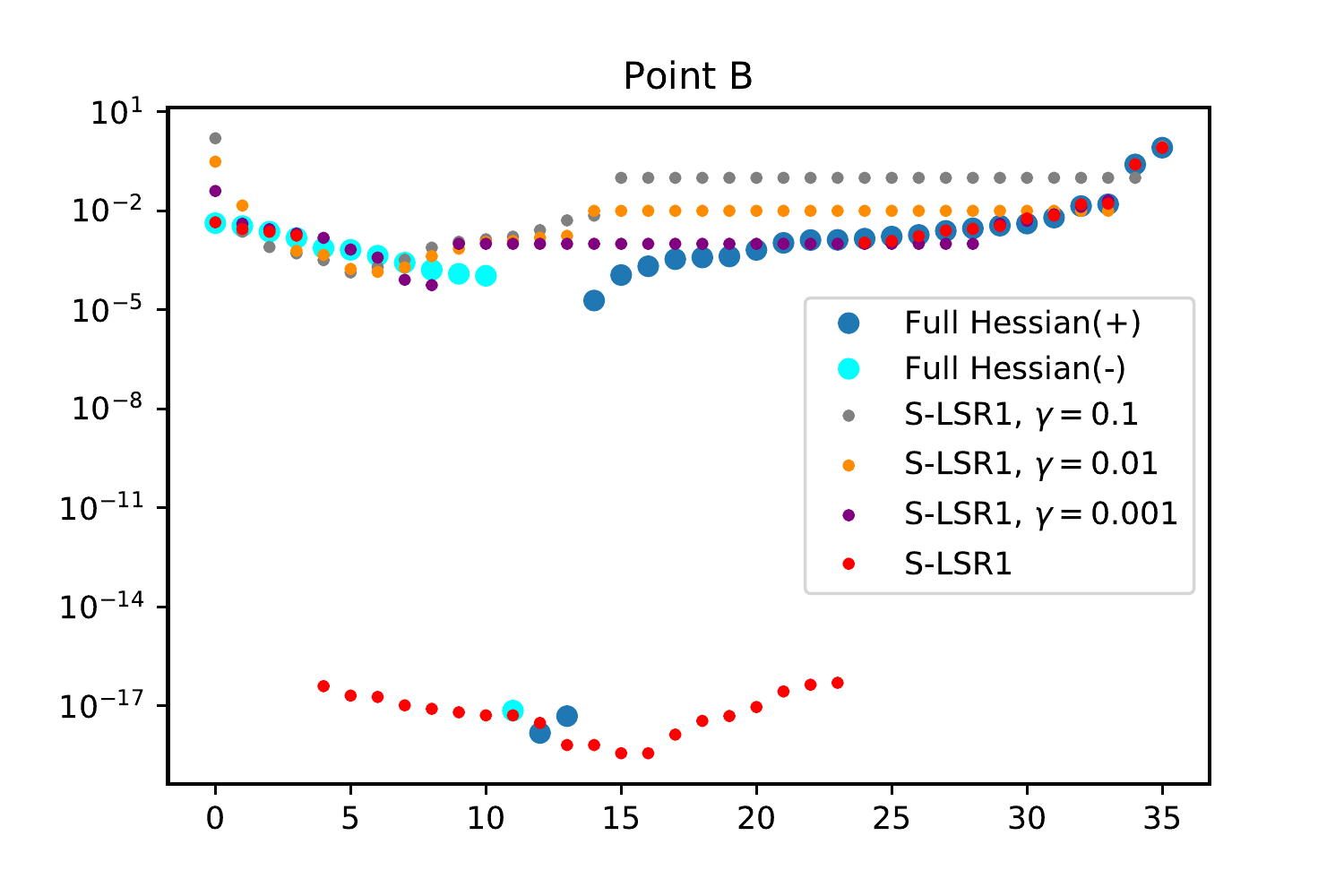}
	\includegraphics[trim=0.6cm 0.7cm 1.0cm 0.7cm,width=0.24\textwidth]{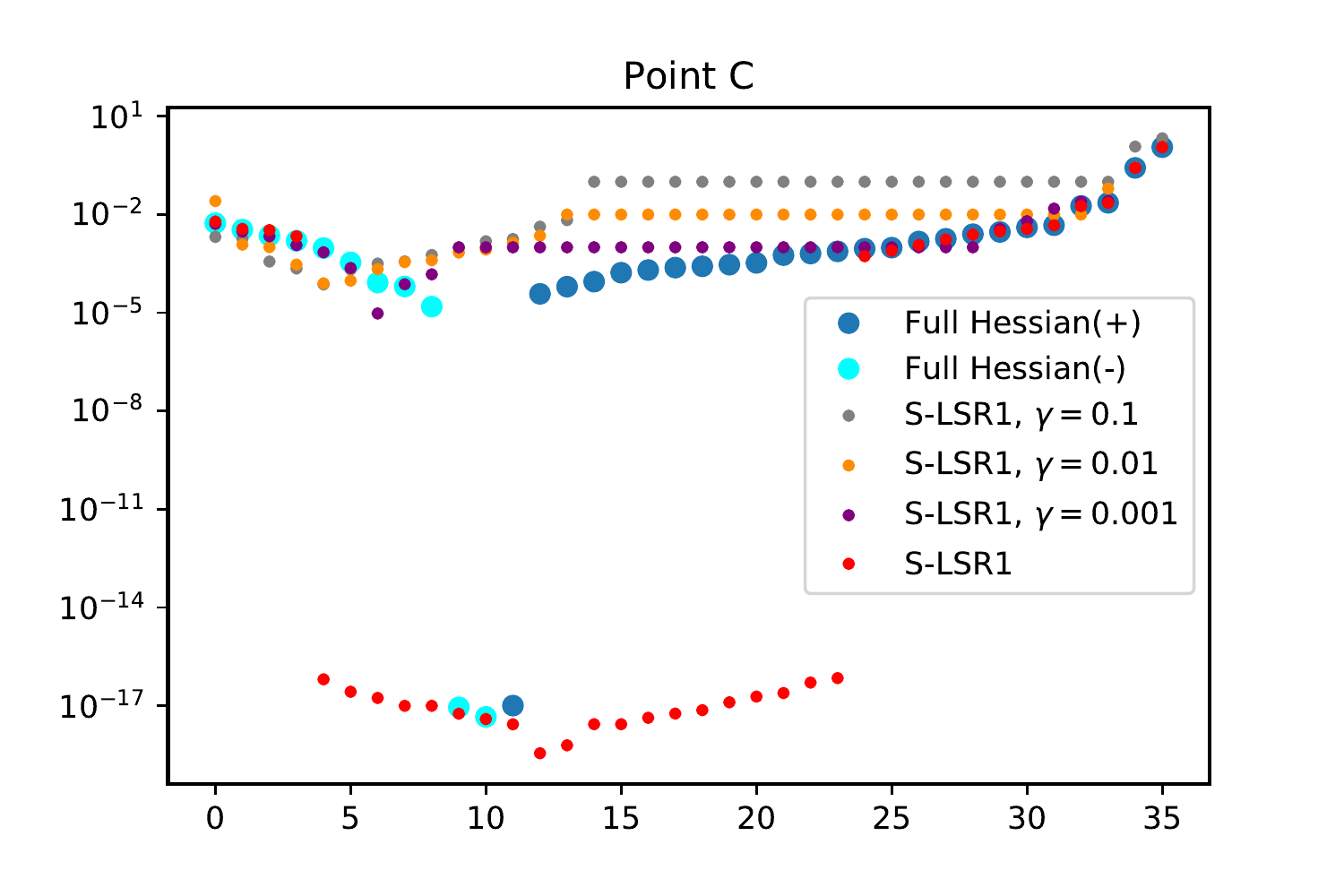}
	\caption{Comparison of the eigenvalues of S-LSR1 for different $\gamma$ (@ A, B, C) for a toy classification problem.
	\label{SmalllvdNet}}
\end{figure}

\subsection{Naive Distributed Implementation of S-LSR1}
In this section, we describe a naive distributed implementation of the S-LSR1 method, where the data is stored across $\mathcal{K}$ machines. At each iteration $k$, we broadcast the current iterate $w_k$ to every worker node. The worker nodes calculate the local gradient, and construct local curvature pairs $S_{k}^i$ and $Y_{k}^i$. The local information is then reduced to the master node to form $\nabla F(w_k)$, $S_k$ and $Y_k$. Next, the SR1 curvature pair condition \eqref{eq:cond_slsr1} is recursively checked on the master node. Given a set of accepted curvature pairs, the master node computes the search direction $p_k$. We should note that the last two steps could potentially be done in a distributed manner at the cost of $m+1$ extra expensive rounds of communication. Finally, given a search direction the trial iterate is broadcast to the worker nodes where the local objective function is computed and reduced to the master node, and a step is taken.
% In this section, we describe a naive distributed implementation of the S-LSR1 method, where the data is stored across $\mathcal{K}$ machines. In order to implement Algorithm \ref{alg:slsr1} in a distributed manner, at each iteration $k$, we broadcast the current iterate $w_k$ to every worker node. The worker nodes then calculate the local objective function and gradient, and construct local curvature pair $S_{k}^i$ and $Y_{k}^i$. The local information is then reduced to the master node to form $F(w_k)$, $\nabla F(w_k)$, $S_k$ and $Y_k$. The SR1 curvature pair condition \eqref{eq:cond_slsr1} is then recursively checked on the master node. Given a set of accepted curvature pairs, the search direction $p_k$ is computed on the master node. We should note that the last two steps could potentially be done in a distributed manner at the cost of $m+1$ extra expensive rounds of communication. Finally, given a search direction the trial iterate is broadcast to the worker nodes where the local objective function is computed and reduced to the master node, and a step is taken.

As is clear, in this distributed implementation of the S-LSR1 method, the amount of information communicated is large, and the amount of computation performed on the master node is significantly larger than that on the worker nodes. Note, all the Hessian vector products, as well as the computations of the $M_k^{-1}$ are performed on the master node. The precise communication and computation details are summarized in Tables \ref{tbl:comm} and \ref{tbl:comp}. 

\section{Efficient Distributed S-LSR1 (DS-LSR1)} 
\label{sec:ds-lsr1}

The naive distributed implementation of  S-LSR1 has several significant deficiencies. We propose a distributed variant of the S-LSR1 method that alleviates these issues, is communication-efficient, has favorable work-load balancing across nodes and is inverse-free and matrix-free. To do this, we leverage the form of the compact representation of the S-LSR1 updating formula ($B_k^{(0)}=0$)
\begin{align}   \label{eq:compactFormVec1_1}
     B_{k+1}v = Y_kM_k^{-1}Y_k^Tv,
\end{align}
and the form of the SR1 condition \eqref{eq:cond_slsr1}. 
% \begin{align}   \label{eq:cond_slsr1_1}
%     | s_{k,j}^T(y_{k,i}-B_k^{(j-1)}s_{k,j}) | \geq \eta \| s_{k,j}\| \| y_{k,i}-B_k^{(j-1)}s_{k,j} \|,
% \end{align}
%for $j=1,\dots,m$. 
We observe the following: one need not communicate the full $S_k$ and $Y_k$ matrices, rather one can communicate $S_k^TY_k$, $S_k^TS_k$ and $Y_k^TY_k$. We now discuss the means by which we: $(1)$ reduce the amount of information communicated and $(2)$ balance the computation across nodes. 

\subsection{Reducing the Amount of Information Communicated}
As mentioned above, communicating  curvature pairs is not necessary; instead one can just communicate inner products of the pairs, reducing the amount of communication from $2md$ to $3m^2$. In this section, we show how this can be achieved, and in fact show that this can be further reduced to $m^2$. 

\paragraph{Construction of $S_k^TS_k$ and $S_k^TY_k$} Since the curvature pairs are scale invariant \cite{berahas2019sqn}, $S_k$ can be any random matrix. Therefore, each worker node can construct this matrix by simply sharing random seeds. In fact, the matrix $S_k^TS_k$ need not be communicated to the master node as the master node can construct and store this matrix. With regards to the $S_k^TY_k$, each worker node can construct local versions of the $Y_k$ curvature pair, $Y_{k}^i$, and send $S_k^TY_{k}^i$ to the master node for aggregation, i.e., $S_k^TY_k = \sfrac{1}{\mathcal{K}}\sum_{i=1}^{\mathcal{K}}S_k^TY_k^i$. Thus, the amount of information communicated to the master node is $m^2$. 

% \begin{wrapfigure}{r}{0.32\textwidth}
% \vspace{-25 pt}
%   \centering
%  	\includegraphics[width=0.31\textwidth]{figs/sketching_log(1).pdf}
%  	\vspace{-8 pt}
% 	\caption{Error in the approximation as a function of the sketch size $m$.
% 	\label{fig:sketch}}
%   \vspace{-25 pt}
% \end{wrapfigure}
\paragraph{Construction of $Y_k^TY_k$} Constructing the matrix $Y_k^TY_k$ in distributed fashion, without communicating local $Y_{k}^i$ matrices, is not that simple. In our communication-efficient method, we propose that the matrix is approximated via sketching \cite{woodruff2014sketching}, using quantities that are already computed, i.e., $Y_k^TY_k \approx Y_k^TS_kS_k^TY_k$. In order for the sketch to be well defined, $S_k \sim \mathcal{N}(0,I/m)$, thus satisfying the conditions of sketching matrices \cite{woodruff2014sketching}. By using this technique, we construct an approximation to $Y_k^TY_k$ with no additional communication. %Figure \ref{fig:sketch} illustrates the dependence of the error on the sketch size. 
Note, the sketch size in our setting is equal to the memory size $m$. We should also note that this approximation is only used in checking the SR1 condition \eqref{eq:cond_slsr1}, which is not sensitive to approximation errors, and not in the Hessian vector products. 

\subsection{Balancing the Computation Across the Nodes}
Balancing the computation across the nodes does not come for free. We propose the use of a few more rounds of communication. The key idea is to exploit the compact representation of the SR1 matrices and perform as much computation as possible on the worker nodes. 

\paragraph{Computing Hessian vector products $B_{k+1}v$} The Hessian vector products \eqref{eq:compactFormVec1_1}, require products between the matrices $Y_k$, $M_k^{-1}$ and a vector $v$. Suppose that the we have $M_k^{-1}$ on the master node, and that the master node broadcasts this information as well as the vector $v$ to the worker nodes. The worker nodes then locally compute $M_k^{-1}(Y_k^i)^Tv$, and send this information back to the master node. The master node then reduces this to form $M_k^{-1}(Y_k)^Tv$, and broadcasts this vector back to the worker nodes. This time the worker nodes compute $Y_k^{i}M_k^{-1}(Y_k)^Tv$ locally, and then this quantity is reduced by the master node; the cost of this communication is $d$. Namely, in order to compute Hessian vector products, the master node performs two aggregations, the bulk of the computation is done on the worker nodes and the communication cost is $m^2 + 2m+2d$. 

\paragraph{Checking the SR1 Condition \ref{eq:cond_slsr1}} As proposed in \cite{berahas2019sqn}, at every iteration condition \eqref{eq:cond_slsr1} is checked recursively by the master node. For each pair in memory, checking this condition amounts to a Hessian vector product as well as the use of inner products of the curvature pairs. Moreover, it requires the computation of $(M_k^{(j)})^{-1} \in \mathbb{R}^{j \times j}$, for $j=1,\dots,m$, where $M_k^{-1} = (M_k^{(m)})^{-1}$. 

\paragraph{Inverse-Free Computation of $M_k^{-1}$} The matrix $M_k^{-1}$ is non-singular \cite{byrd1994}, depends solely on inner products of the curvature pairs, and is used in the the computation of Hessian vector products \eqref{eq:compactFormVec1_1}. This matrix is constructed recursively (its dimension grows with the memory) by the master node as condition \eqref{eq:cond_slsr1} is checked. We propose an inverse-free approach for constructing this matrix. Suppose we have the matrix $(M_k^{(j)})^{-1}$, for some $j=1,\dots,m-1$, and that the new curvature pair $(s_{k,j+1},y_{k,j+1})$ satisfies \eqref{eq:cond_slsr1}. One can show that 
\begin{equation*}%\label{eq:MinvRec1}
(M_k^{(j+1)})^{-1}=
\left[
\begin{array}{c:c}
(M_k^{(j)})^{-1}+\zeta(M_k^{(j)})^{-1}uv^T(M_k^{(j)})^{-1} & -\zeta(M_k^{(j)})^{-1} u \\
\hdashline
- \zeta v^T(M_k^{(j)})^{-1} & \zeta
\end{array}
\right]
\end{equation*}
where $\zeta = \sfrac{1}{c - v^T(M_k^{(j)})^{-1}u}$, $v^T = s_{k,j+1}^TY_{k,1:l}$ and $Y_{k,1:l}=[y_{k,1},\dots,y_{k,l}]$ for $l\leq j$, $ u = v$, and $c =s_{k,j+1}^T y_{k,j+1}$. We should note that the matrix $(M_k^{(1)})^{-1}$ is a singleton. Consequently, constructing $(M_k^{(j)})^{-1}$ in an inverse-free manner allows us to compute Hessian vector products and check condition \eqref{eq:cond_slsr1} efficiently.

\subsection{The Distributed S-LSR1 (DS-LSR1) Algorithm}
Pseudo-code for our proposed distributed variant of the S-LSR1 method and the curvature pair sampling procedure are given in Algorithms \ref{alg:dslsr1} and \ref{alg:calSYdist}, respectively. Right arrows denote broadcast steps and left arrows denote reduce steps. For brevity we omit the details of the distributed CG-Steihaug algorithm (Line 5, Algorithm \ref{alg:dslsr1}), but note that it is a straightforward adaptation of \cite[Algorithm 7.2]{nocedal_book} using quantities described above computed in distributed fashion.%\vspace{-15pt}

\begin{algorithm}[H]
{\small 
\caption{Distributed Sampled LSR1 (DS-LSR1)}
  \label{alg:dslsr1}
 {\bf Input:} $w_{0}$ (initial iterate), $\Delta_0$ (initial trust region radius), 
$m$ (memory). 

\textbf{Master Node:} \hfill \textbf{Worker Nodes (\pmb{$i=1,2,\dots,\mathcal{K}$}):} 
  \begin{algorithmic}[1]
  \For {$k=0,1,2,...$}
    \State {\color{green!40!black}{\it \textbf{Broadcast:}}} $w_k$ \hfill \textcolor{NavyBlue}{$\pmb{\longrightarrow}$} \hfill Compute  $F_i(w_k)$, $\nabla F_i(w_k)$
    \State {\color{orange!50!black}{\it \textbf{Reduce:}}} $F_i(w_k) $, $\nabla F_i(w_k)$ to $F(w_k) $, $\nabla F(w_k)$ \hfill \textcolor{NavyBlue}{$\pmb{\longleftarrow}$} \hfill $\ \ \ $
    \State Compute new $(M^{-1}_k,Y_k,S_k)$ pairs via Algorithm \ref{alg:calSYdist}
    \State Compute $p_k$ via CG-Steihaug \cite[Algorithm 7.2]{nocedal_book}
    \State  {\color{green!40!black}{\it \textbf{Broadcast:}}} $p_k$, $M^{-1}_k$ \quad \textcolor{NavyBlue}{$\pmb{\longrightarrow}$} \hfill Compute $M_k^{-1}(Y_k^i)^T p_k $, $\nabla F_i(w_k)^Tp_k$, $F_i(w_k +p_k)$  
    \State {\color{orange!50!black}{\it \textbf{Reduce:}}} $M_k^{-1}(Y_k^i)^T p_k $, $\nabla F_i(w_k)^Tp_k$, $F_i(w_k +p_k)$ to $M_k^{-1}Y_k^T p_k $, $\nabla F(w_k)^Tp_k$, $F(w_k +p_k)$ \hfill \textcolor{NavyBlue}{$\pmb{\longleftarrow}$} \hfill $\ \ \ $
    \State {\color{green!40!black}{\it \textbf{Broadcast:}}} $M_k^{-1}Y_k^T p_k $ \hfill \textcolor{NavyBlue}{$\pmb{\longrightarrow}$} \hfill Compute $(Y_k^{i})^TM_k^{-1}Y_k^T p_k $
    \State {\color{orange!50!black}{\it \textbf{Reduce:}}} $(Y_k^{i})^TM_k^{-1}Y_k^{i} p_k $ to $B_kp_k=(Y_k)^TM_k^{-1}Y_k p_k $  \hfill \textcolor{NavyBlue}{$\pmb{\longleftarrow}$} \hfill $\ \ \ $ 
    \State Compute $\rho_k = \frac{F(w_k) - F(w_k + p_k)}{m_k(0) - m_k(p_k)}$
     \State \textbf{if } $\rho_k \geq \eta_1$ \textbf{then } Set $w_{k+1} = w_k + p_k$ \textbf{ else } Set $w_{k+1} = w_k$
    \State $\Delta_{k+1} = \texttt{adjustTR}(\Delta_{k},\rho_k)$ \cite[Appendix B.3]{berahas2019sqn}
\EndFor
  \end{algorithmic}
  }
\end{algorithm}

 \begin{algorithm}[]
{ 
\small
\caption{Construct new $(S_k,Y_k)$ curvature pairs}
  \label{alg:calSYdist}
 {\bf Input:} $w_k$ (iterate), $m$ (memory), $S_k = [\;]$, $Y_k =[\;]$ (curvature pair containers). 

\textbf{Master Node:} \hfill \textbf{Worker Nodes (\pmb{$i=1,2,\dots,\mathcal{K}$}):} 
  \begin{algorithmic}[1]
  \State {\color{green!40!black}{\it \textbf{Broadcast:}}} $\bar{S}_k$ and $w_k$ \hfill  \textcolor{NavyBlue}{$\pmb{\longrightarrow}$}  \hfill Compute   $\bar{Y}_{k,i} = \nabla^2 F_{i}(w_k)\bar{S}_k$
  \State {\color{orange!50!black}{\it \textbf{Reduce:}}} $\bar{S}_k^T\bar{Y}_{k,i}$ to  $\bar{S}_k^T\bar{Y}_k$ and $\bar{Y}_k^T\bar{S}_k\bar{S}_k^T\bar{Y}_k$ \hfill  \textcolor{NavyBlue}{$\pmb{\longleftarrow}$} \hfill Compute   $\bar{S}_k^T\bar{S}_k$ and $\bar{S}_k^T\bar{Y}_{k,i}$
  \State Check the SR1 condition \eqref{eq:cond_slsr1} and construct $M_k^{-1}$ recursively using
  
   $\bar{S}_k^T\bar{S}_k$, $\bar{S}_k^T\bar{Y}_k$ and $\bar{Y}_k^T\bar{Y}_k$ and construct list of accepted pairs $S_k$ and $Y_k$
  %\State Construct the accepted list $S_k$ and $Y_k$
  \State {\color{green!40!black}{\it \textbf{Broadcast:}}} the list of accepted curvature pairs
  \end{algorithmic}
  {\bf Output:} $M^{-1}$, $Y_k$, $S_k$
  }
\end{algorithm}

\subsection{Complexity Analysis - Comparison of Methods}

\begin{wraptable}{r}{0.42\textwidth}
\vspace{-22pt}
\caption{Details of quantities communicated and computed.}
\vspace{-5pt}
\label{tbl:quant_details}
\centering
\begin{scriptsize}
\begin{tabular}{@{}lc@{}}\toprule
\textbf{Variable} & \textbf{Dimension} \\ \midrule
 $w_k, \nabla F(w_k), \nabla F_i(w_k)$ \\ $p_k, Y_{k,i}M_k^{-1}Y_k^Tp_k,B_kd$ &  $d \times 1$
\\ \hdashline  
 $F(w_k), F_i(w_k)$ &  $1$
\\ \hdashline
 $S_k,S_{k,i}, Y_k,Y_{k,i}$&  $d \times m$
\\ \hdashline  
$S_k^TY_{k,i},S_{k,i}^TY_{k,i},M_k^{-1}$&  $m \times m$ 
\\\hdashline 
 $M_k^{-1}Y_{k,i}^Tp_k$ &  $m \times 1$
\\  
\bottomrule 
\end{tabular}
\end{scriptsize}
\vspace{-25pt}
\end{wraptable}
% \begin{wraptable}{r}{0.35\textwidth}
% \vspace{-25pt}
% \caption{Details of quantities communicated and computed.}
% \label{tbl:quant_details}
% \centering
% \begin{small}
% \begin{tabular}{@{}ll@{}}\toprule
% \textbf{Variable} & \textbf{Dimension} \\ \midrule
%  $w_k$&  $d \times 1$
% \\ \hdashline  
%  $F(w_k), F_i(w_k)$ &  $1$
%  \\ \hdashline  
%  $\nabla F(w_k), \nabla F_i(w_k)$ &  $d \times 1$
% \\ \hdashline
% $p_k$ &  $d \times 1$
% \\ \hdashline
%  $S_k,S_{k,i}$&  $d \times m$
% \\ \hdashline  
% $Y_k,Y_{k,i}$&  $d \times m$ 
% \\\hdashline  
% $S_k^TY_{k,i},S_{k,i}^TY_{k,i}$&  $m \times m$ 
% \\\hdashline 
%  $M_k^{-1}$&  $m \times m$
%  \\\hdashline  
%  $B_kd$&  $d \times 1$
%  \\\hdashline  
%  $M_k^{-1}Y_{k,i}^Tp_k$ &  $m \times 1$
%  \\\hdashline  
%  $Y_{k,i}M_k^{-1}Y_k^Tp_k$ &  $d \times 1$
%  \\\hdashline
%  $M_k^{-1}$&  $m \times m$
% \\  
% \bottomrule 
% \end{tabular}
% \end{small}
% \vspace{-25pt}
% \end{wraptable}
We compare the complexity of a naive distributed implementation of S-LSR1 and DS-LSR1. Specifically, we discuss the amount of information communicated at every iteration and the amount of computation performed by the nodes. Tables \ref{tbl:comm} and \ref{tbl:comp} summarize the communication and computation costs, respectively, and Table \ref{tbl:quant_details} summarizes the details of the quantities presented in the tables. 

As is clear from Tables \ref{tbl:comm} and \ref{tbl:comp} the amount of information communicated in the naive implementation ($2md+d+1$) is significantly larger than that in the DS-LSR1 method ($m^2+2d+2m+1$). Note, $m\ll d$. This can also be seen in Figure \ref{fig:float_comm} where we show for different dimension $d$ and memory $m$ the number of floats communicated at every iteration. To put this into perspective, consider a training problem where $d = 9.2M$ (e.g., VGG11 network \cite{simonyan2014very}) and $m=256$, DS-LSR1 and naive DS-LSR1 need to communicate $0.0688 \,GB$ and $8.8081 \,GB$, respectively, per iteration. In terms of computation, it is clear that in the naive approach the amount of computation is not balanced between the master and worker nodes, whereas for DS-LSR1 the quantities are balanced. 

\begin{minipage}{0.65\textwidth}
\begin{table}[H]
\caption{Communication Details.}
\label{tbl:comm}
\centering
\begin{scriptsize}
\begin{tabular}{@{}lcc@{}}\toprule
\textbf{} & \textcolor{blue}{\textbf{Naive DS-LSR1}} & \textcolor{red}{\textbf{DS-LSR1}} \\ \midrule
\color{green!40!black}{\it \textbf{Broadcast:}} & $w_k$ & $w_k, p_k, M^{-1}$ 
\\ \hdashline  
\color{orange!50!black}{\it \textbf{Reduce:}} & \shortstack{$\nabla F_i(w_k), F_i(w_k),$ \\ $S_{k,i}, Y_{k,i}$} & \shortstack{$\nabla F_i(w_k), F_i(w_k), S_k^TY_{k,i},$ \\ $Y_{k,i}M_k^{-1}Y_{k,i}p_k, M_k^{-1}Y_{k,i}^Tp_k$}  
\\
\bottomrule 
\end{tabular}
\end{scriptsize}
\end{table}
\vspace{-50 pt}
\begin{table}[H]
\caption{Computation Details.}
\label{tbl:comp}
\centering
\begin{scriptsize}
\begin{tabular}{@{}lcc@{}}\toprule
\textbf{} & \textcolor{blue}{\textbf{Naive DS-LSR1}} & \textcolor{red}{\textbf{DS-LSR1}} \\ \midrule
\color{blue!50!white}{\it \textbf{Worker:}} & $\nabla F_i(w_k), F_i(w_k), Y_{k,i}$ & \shortstack{$\nabla F_i(w_k), F_i(w_k), Y_{k,i},S_{k,i}^TY_{k,i}$ \\ $M_k^{-1}Y_{k,i}^Tp_k, Y_{k,i}M_k^{-1}Y_k^Tp_k,$ CG} 
\\ \hdashline  
\color{red!20!white}{\it \textbf{Master:}} & $M_k^{-1}, w_{k+1}, B_kd$, CG & $M_k^{-1}, w_{k+1}$ 
\\
\bottomrule 
\end{tabular}
\end{scriptsize}
\end{table}

\end{minipage}
\begin{minipage}{0.34\textwidth}
\begin{figure}[H]
\vspace{10 pt}
	\centering

 	\includegraphics[width=0.95\textwidth]{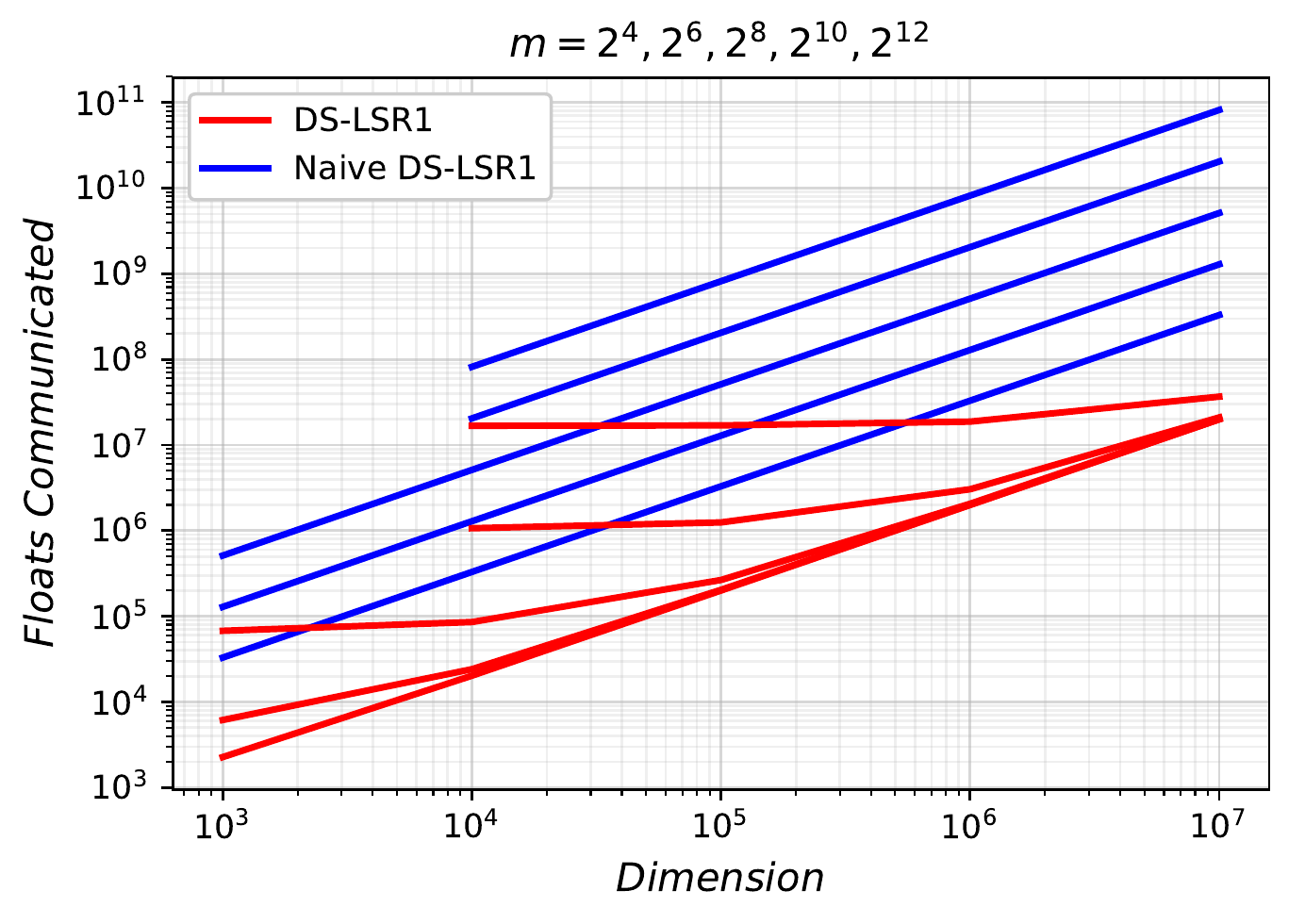}
 	\vspace{-5 pt}
 	\caption{Number of floats communicated per iteration for different dimension $d$ and memory size $m$. \label{fig:float_comm}}
\end{figure}
\end{minipage}

%\newpage 
\section{Numerical Experiments}
\label{sec:num_res}
% In this section, we present a numerical investigation of the proposed DS-LSR1 method.\footnote{\tiny{All algorithms are implemented in Python (PyTorch library), using the MPI for Python distributed environment. The experiments were conducted on XSEDE clusters using GPU nodes. Each physical node includes 4 K80 GPUs, and each MPI process is assigned to a distinct GPU, i.e., 4 MPI processes for each node. Code will be released upon publication of the paper.}} 
The goals of this section are threefold: $(1)$ To illustrate the scaling properties of the method and compare it to the naive implementation (Figures \ref{fig:weak} \& \ref{fig:strong}); $(2)$ To deconstruct the main computational elements of the method and show how they scale in terms of memory (Figure \ref{fig:comp_scaling}); and $(3)$ To illustrate the performance of DS-LSR1 on a neural network training task (Figure \ref{fig:perf}). We should note upfront that the goal of this section is not to achieve state-of-the-art performance and compare against algorithms that can achieve this, rather to show that the method is communication efficient and scalable.\footnote{\scriptsize{All algorithms were implemented in Python (PyTorch library), using the MPI for Python distributed environment. The experiments were conducted on XSEDE clusters using GPU nodes. Each physical node includes 4 K80 GPUs, and each MPI process is assigned to a distinct GPU. Code available at: \href{https://github.com/OptMLGroup/DSLSR1}{https://github.com/OptMLGroup/DSLSR1}.}}

\subsection{Scaling}

%\begin{wrapfigure}{r}{0.46\textwidth}
%\vspace{-20 pt}
%  \centering
%	\includegraphics[width=0.45\textwidth]{figs/bs_all_error_band_1Layer(3).pdf}
%	
% 	\includegraphics[width=0.45\textwidth]{figs/bs_all_error_band_7Layer(2).pdf}
% 	\vspace{-10 pt}
%	\caption{\textbf{Weak Scaling:} Time/iteration (sec) vs \# of variables; \texttt{Shallow} (left), \texttt{Deep} (right).
%	\label{fig:weak}}
%  \vspace{-22 pt}
%\end{wrapfigure}
\paragraph{Weak Scaling} We considered two different types of networks: $(1)$ \texttt{Shallow} (one hidden layer), and $(2)$ \texttt{Deep} (7 hidden layers), and for each varied the number of neurons in the layers (MNIST dataset \cite{lecun1998gradient}, memory $m=64$). Figure \ref{fig:weak} shows the time per iteration for DS-LSR1 for different number of variables and batch sizes.
\begin{figure}[]
	\centering
  	\includegraphics[width=0.45\textwidth]{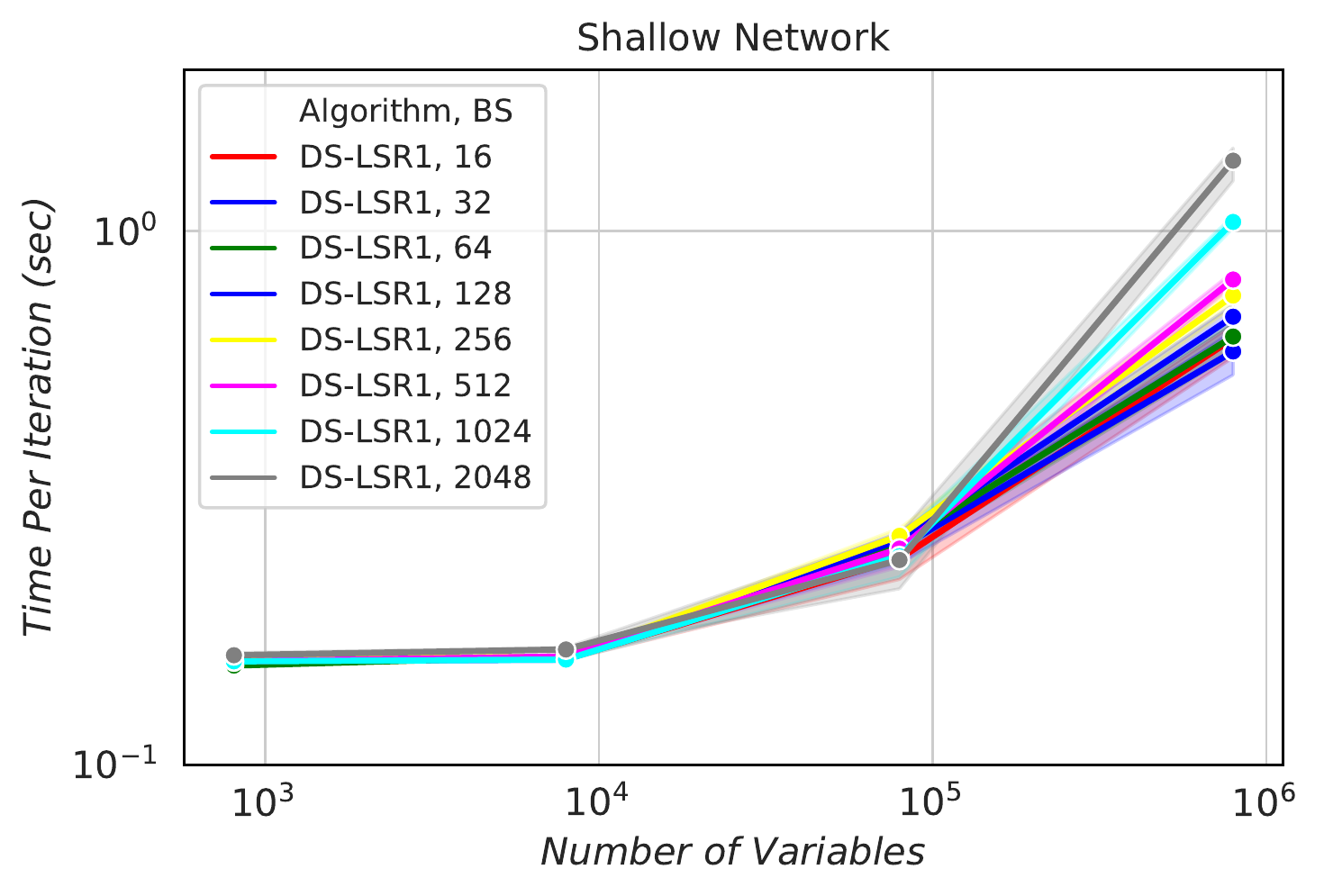}
 	\includegraphics[width=0.45\textwidth]{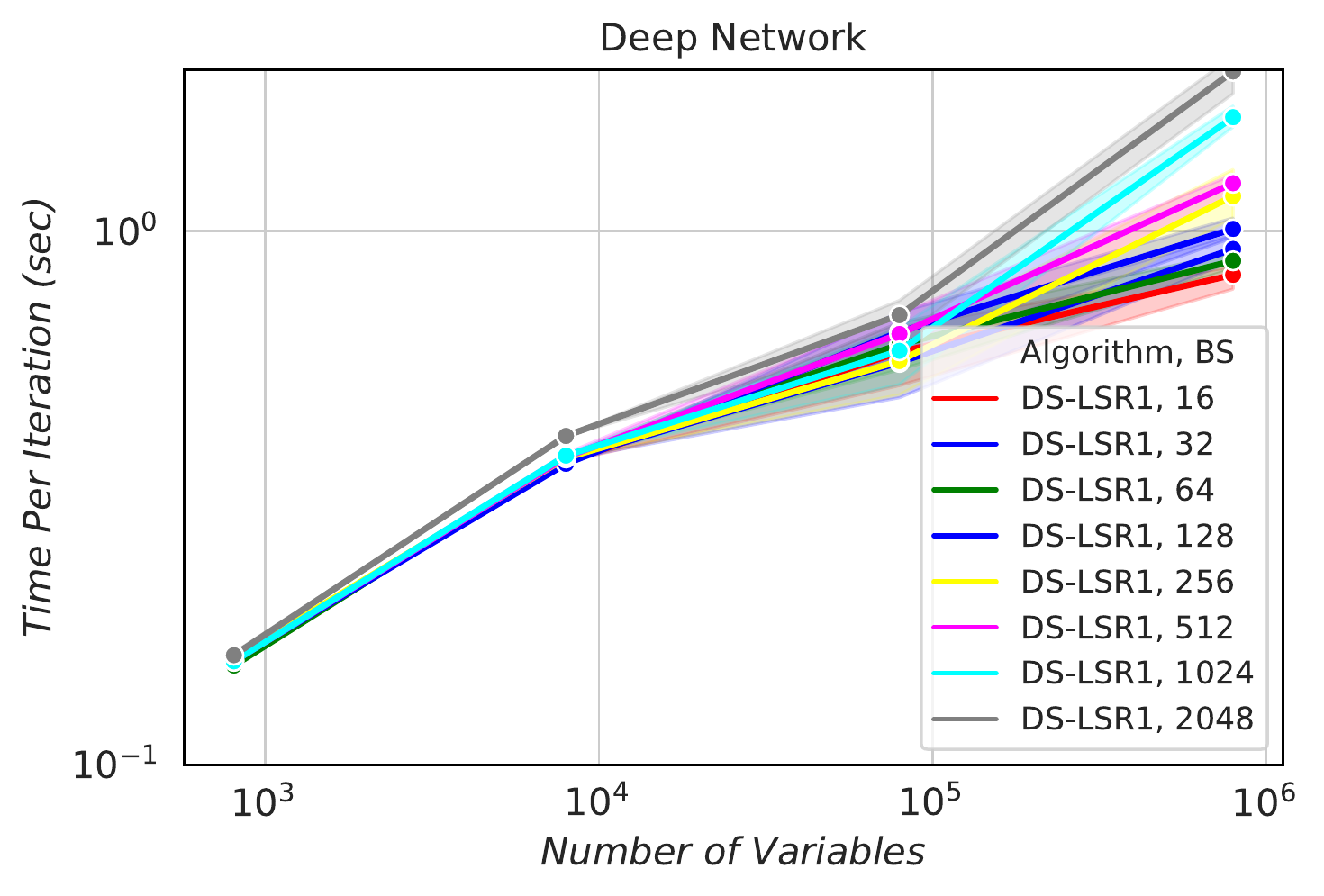}
    \vspace{-5 pt}
	\caption{\textbf{Weak Scaling:} Time/iteration (sec) vs \# of variables; \texttt{Shallow} (left), \texttt{Deep} (right).
	\label{fig:weak}}
\end{figure}
% \vspace{-5pt}
% \begin{minipage}{0.60\textwidth}
% \begin{figure}[H]
% 	\centering
% 	\includegraphics[width=0.49\textwidth]{figs/bs_all_error_band_1Layer(3).pdf}
%  	\includegraphics[width=0.49\textwidth]{figs/bs_all_error_band_7Layer(2).pdf}
%  	\vspace{-17 pt}
% 	\caption{\textbf{Weak Scaling:} Time per iteration (sec) versus number of variables for \texttt{Shallow} (left) and \texttt{Deep} (right) networks.
% 	\label{fig:weak}}
% \end{figure}
% \end{minipage}
% \hspace{2 pt}
% \begin{minipage}{0.30\textwidth}
% \begin{figure}[H]
% 	\centering
%  	%\includegraphics[width=0.32\textwidth]{images/strong_speedup_mmr64-QR.pdf}
%  	\includegraphics[width=0.98\textwidth]{figs/strong_speedup_mmr256_QR.pdf}
%  	\vspace{-17 pt}
% 	\caption{\textbf{Strong Scaling:} Relative speedup versus number of nodes. \label{fig:strong}}
% \end{figure}
% \end{minipage}
%\begin{wrapfigure}{r}{0.46\textwidth}
%\vspace{-20 pt}
%  \centering
% 	\includegraphics[width=0.45\textwidth]{figs/strong_speedup_mmr256_QR.pdf}
%	\caption{\textbf{Strong Scaling:} Relative speedup vs \# of nodes. \label{fig:strong}}
%  \vspace{-24 pt}
%\end{wrapfigure}
\paragraph{Strong Scaling} We fix the problem size (LeNet, CIFAR10, $d=62006$ \cite{lecun1998gradient}), vary the number of nodes and measure the speed-up achieved. Figure \ref{fig:strong} illustrates the relative speedup (normalized speedup of each method with respect to the performance of that method on a single node) of the DS-LSR1 method and the naive variant for  $m=256$. The DS-LSR1 method achieves near linear speedup as the number of nodes increases, and the speedup is better than that of the naive approach. We should note that the times of our proposed method are lower than the respective times for the naive implementation. The reasons for this are: $(1)$ DS-LSR1 is inverse free, and $(2)$ the amount of information communicated is significantly smaller.
\begin{figure}[]
	\centering
	\includegraphics[width=0.45\textwidth]{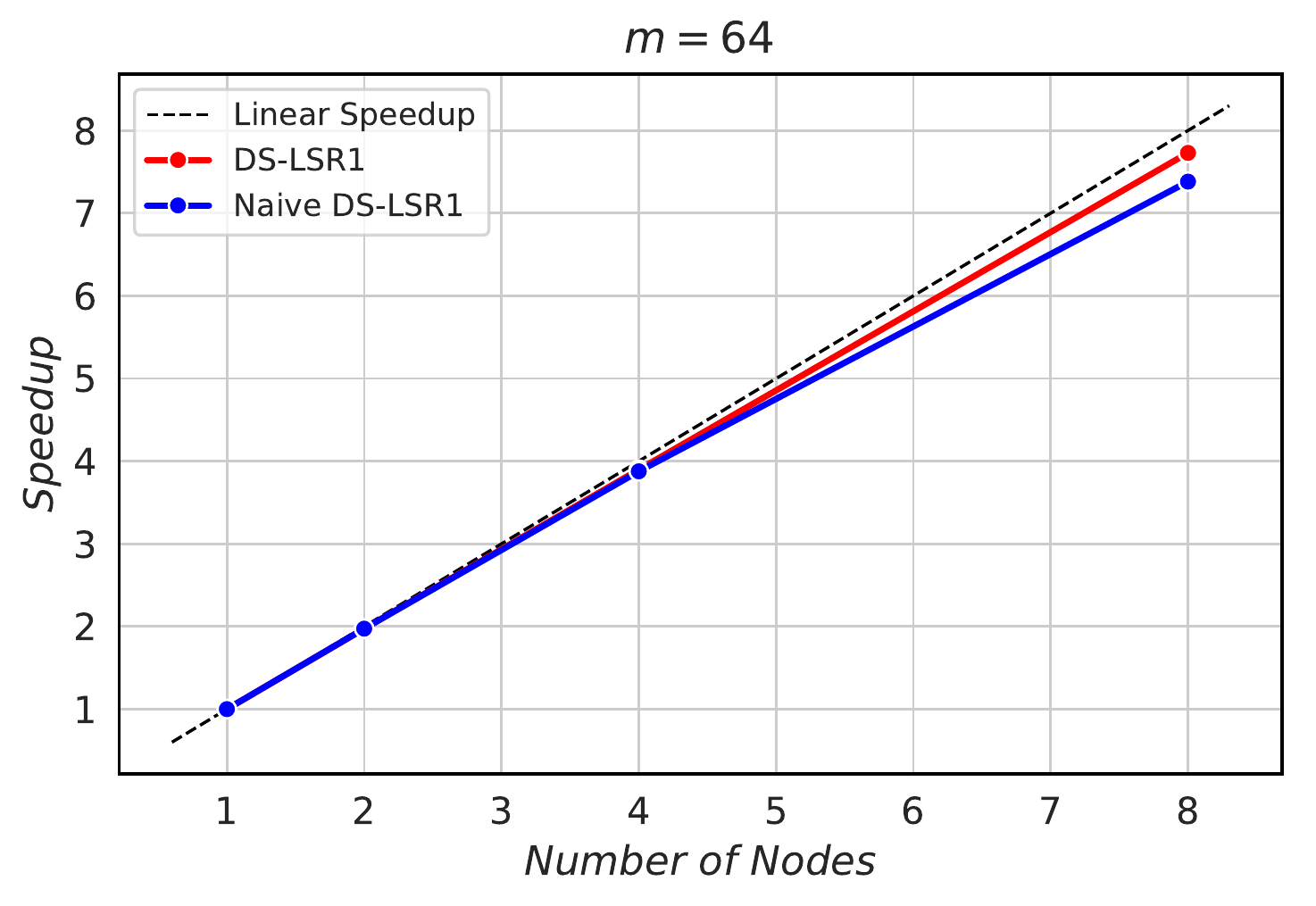}
  	\includegraphics[width=0.45\textwidth]{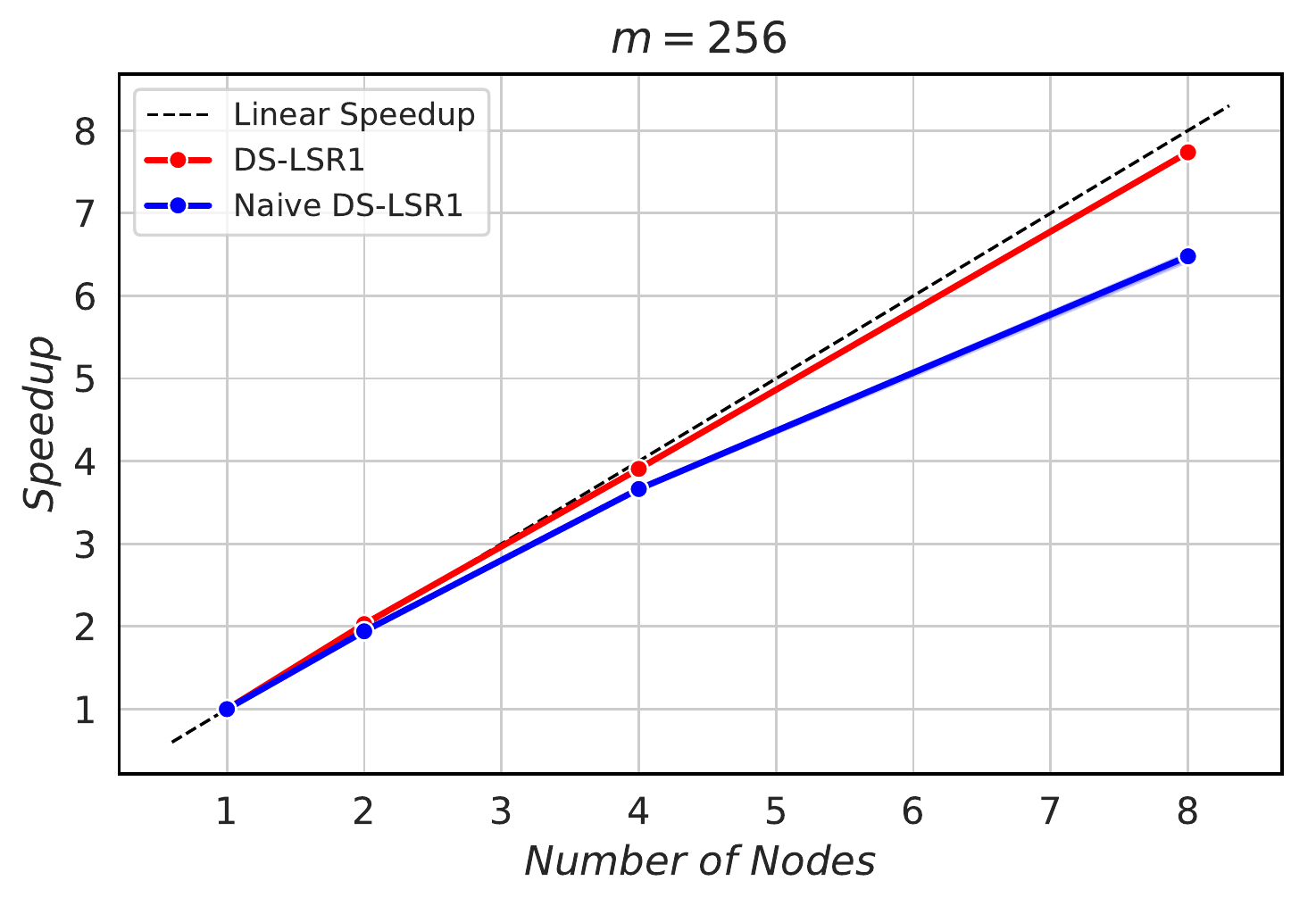}
    \vspace{-5 pt}
	\caption{\textbf{Strong Scaling:} Relative speedup vs \# of nodes. \label{fig:strong}}
\end{figure}

\paragraph{Scaling of Different Components of DS-LSR1} We deconstruct the main components of the DS-LSR1 method and illustrate the scaling (per iteration) with respect to memory size. Figure \ref{fig:comp_scaling} shows the scaling for: $(1)$ reduce time; $(2)$ total time; $(3)$ CG time; $(4)$ time to sample $S$, $Y$ pairs. For all these plots, we ran $10$ iterations, averaged the time and also show the variability. As is clear, our proposed method has lower times for all components of the algorithm. We attribute this to the aforementioned reasons. %Here we deconstruct the main components of the DS-LSR1 method and illustrate the scaling with respect to memory. Specifically, Figure \ref{fig:comp_scaling} shows the scaling for: $(1)$ reduce time/iteration; $(2)$ time/iteration; $(3)$ CG time/iteration; $(4)$ time to sample $S$, $Y$ pairs/iteration. For all these plots, we ran $10$ iterations and averaged the time, and also show the variability. As is clear for the figure, our proposed method has lower times for all components of the algorithm. Again, we attribute this to the fact that our approach: $(1)$ requires less information exchange (communication) per iteration; and $(2)$ is inverse-free. \vspace{-20 pt}

\begin{figure}[]
	\centering
  	\includegraphics[width=0.44\textwidth]{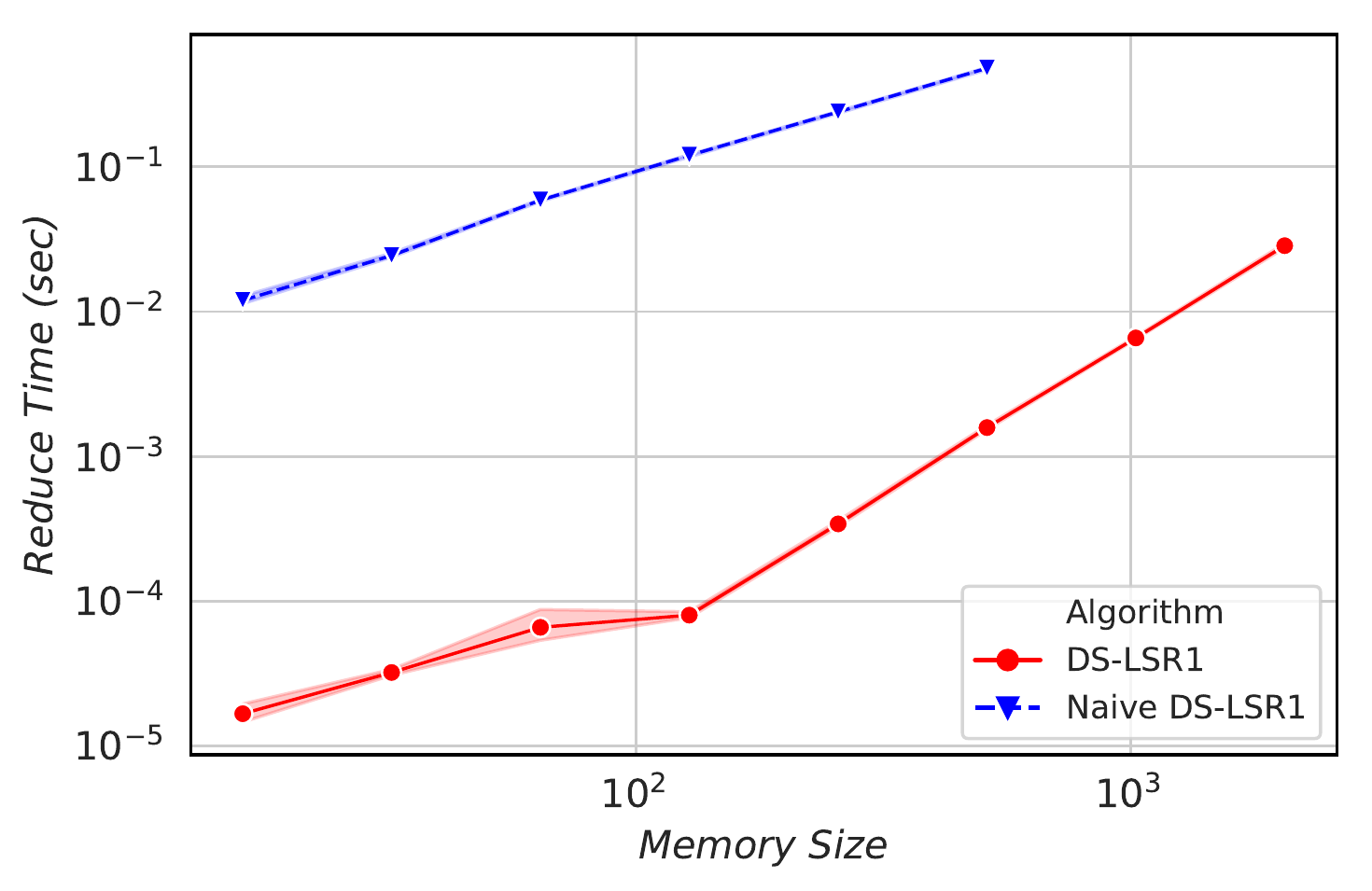}
 	\includegraphics[width=0.44\textwidth]{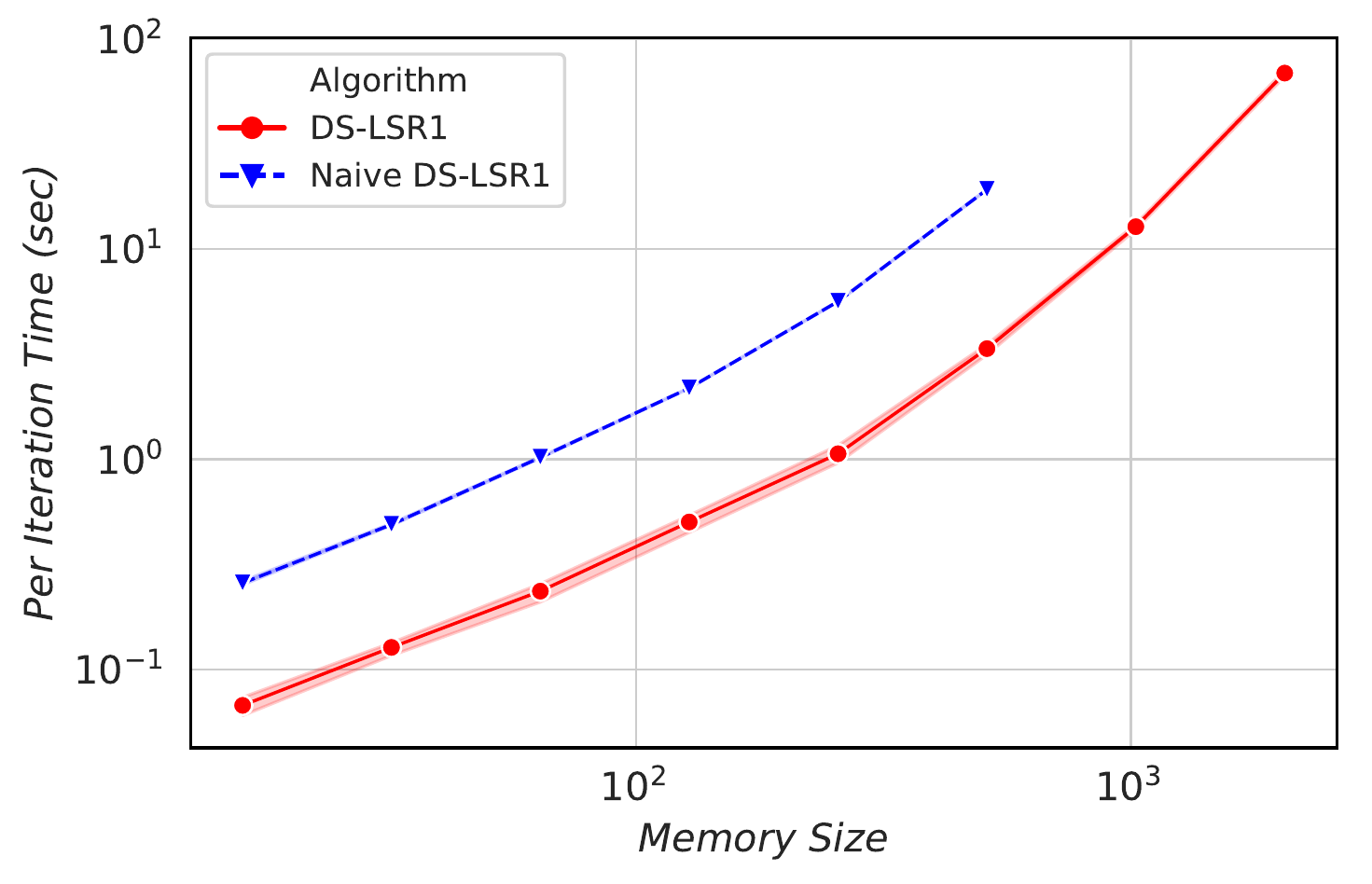} 
	
  	\includegraphics[width=0.44\textwidth]{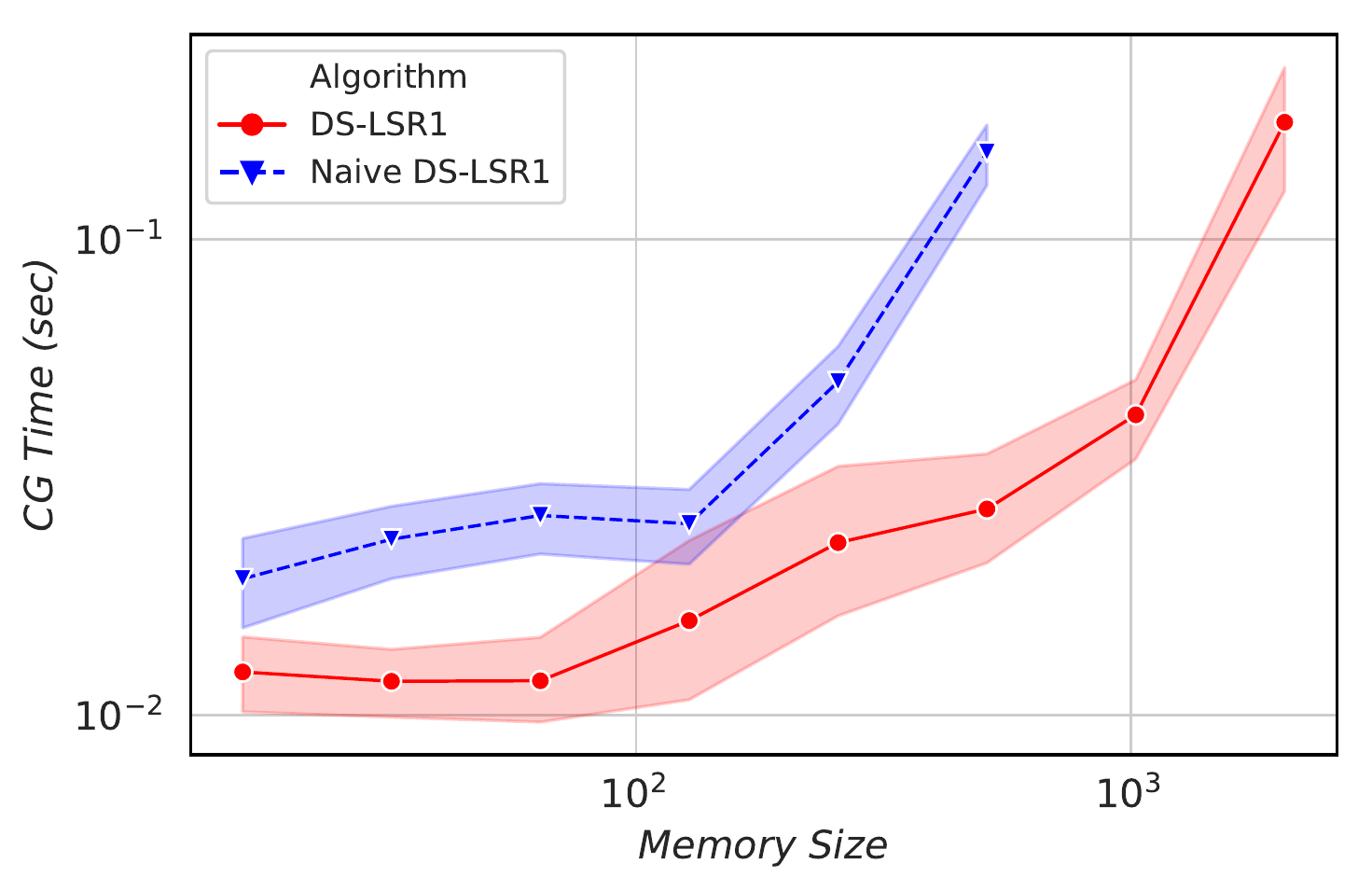}
  	\includegraphics[width=0.44\textwidth]{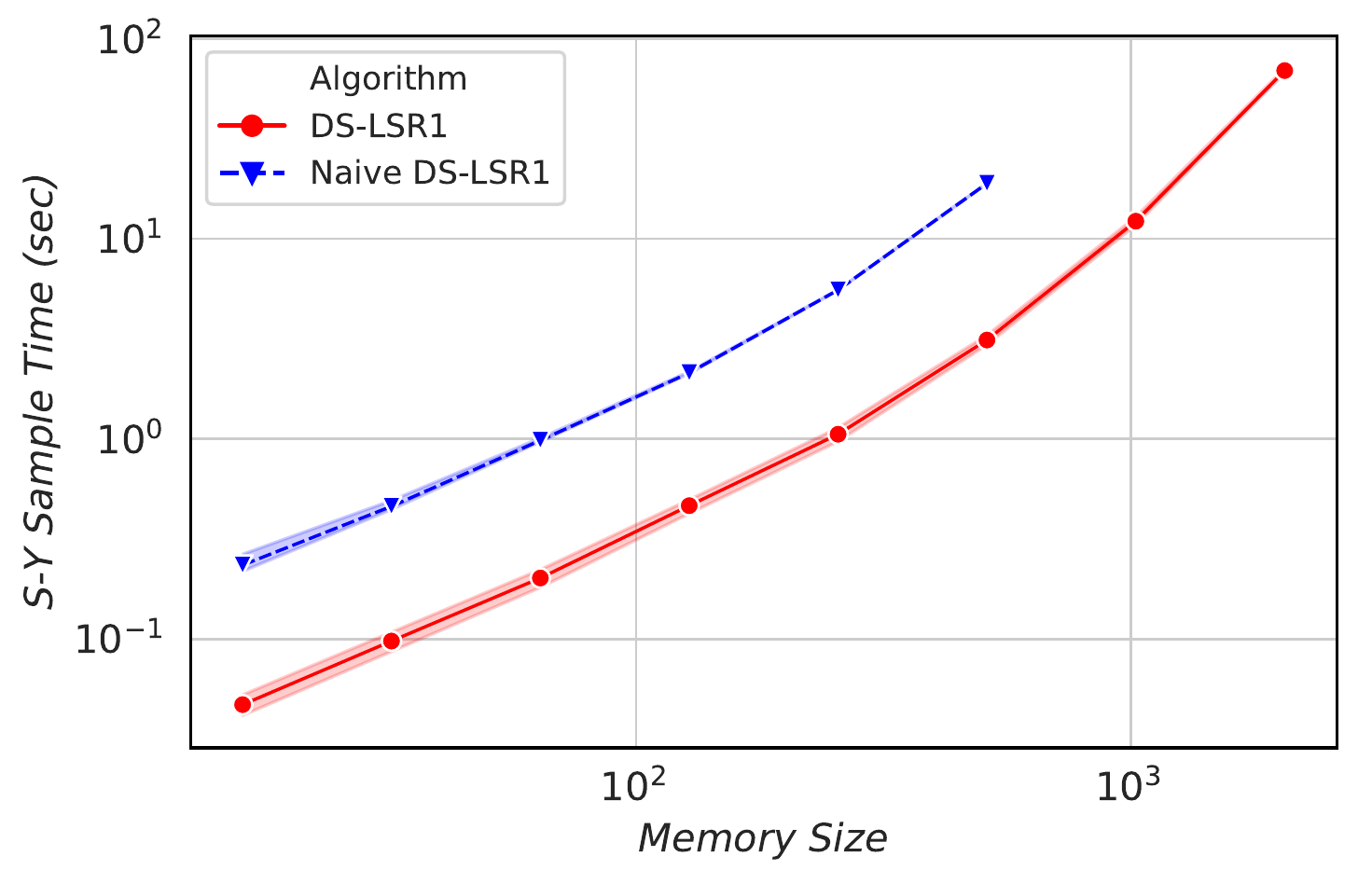}
    \vspace{-5 pt}
	\caption{Time (sec) for different components of DS-LSR1 with respect to memory. \label{fig:comp_scaling}}
\end{figure}

% \begin{wrapfigure}{r}{0.33\textwidth}
% \vspace{-25 pt}
%   \centering
%  	%\includegraphics[width=0.4\textwidth]{comp_red_dim_heat_shallow_btc16.pdf}
%   	\includegraphics[width=0.32\textwidth]{figs/comp_red_dim_heat_shallow_btc16_log(1).pdf}
% 	\caption{Ratios of reduce times per iteration with respect to batch size and dimension.
% 	\label{redMem}}
% 	\vspace{-25 pt}
% \end{wrapfigure}
% \textcolor{blue}{To further highlight the efficiency of our proposed method, in terms of communications, we plot the ratio of the reduce time per iteration of DS-SLR1  to the reduce time per iteration of the naive distributed implementation in Figure \ref{redMem}. For these experiments we set the memory size to $m=64$. As is clear, the reduce time for DS-LSR1 is significantly smaller than that of the naive approach. As expected, this is especially true when the number of variables in the problem $d$ is large.}

\subsection{Performance of DS-LSR1}

In this section, we show the performance of DS-LSR1 on a neural network training task; LeNet \cite{lecun1998gradient}, CIFAR10, $n=50000$, $d=62006$, $m=256$.  Figure \ref{fig:perf} illustrates the training accuracy in terms of wall clock time and amount of data (GB) communication (left and center plots), for different number of nodes. As expected, when using larger number of compute nodes training is faster. Similar results were obtained for testing accuracy. We also plot the performance of the naive implementation (dashed lines) in order to show that: $(1)$ the accuracy achieved is comparable, and $(2)$ one can train faster using our proposed method.
%In this section, we show the performance of the DS-LSR1 method on a neural network training task; LeNet \cite{lecun1998gradient}, CIFAR10, $n=50000$, $d=62006$. For this experiment we set memory to $m=256$. In Figure \ref{fig:perf}, we illustrate the training accuracy in terms of wall clock time and amount of data (GB) communication (left and center plots, respectively), for different number of nodes. As expected, when using larger number of compute nodes training is faster, i.e., given a fixed time budget, the accuracy achieved when using more nodes is higher. Similar results were obtained for testing accuracy; see Appendix \ref{sec:app_perf}. We also plot, the performance of the naive implementation. Firstly, to show that the accuracy achieved is comparable, and thus the two approaches are identical. And, secondly, to show that one can train faster using our proposed method.%\vspace{-20 pt}

\begin{figure}[]
		\centering
			\includegraphics[width=0.44\textwidth]{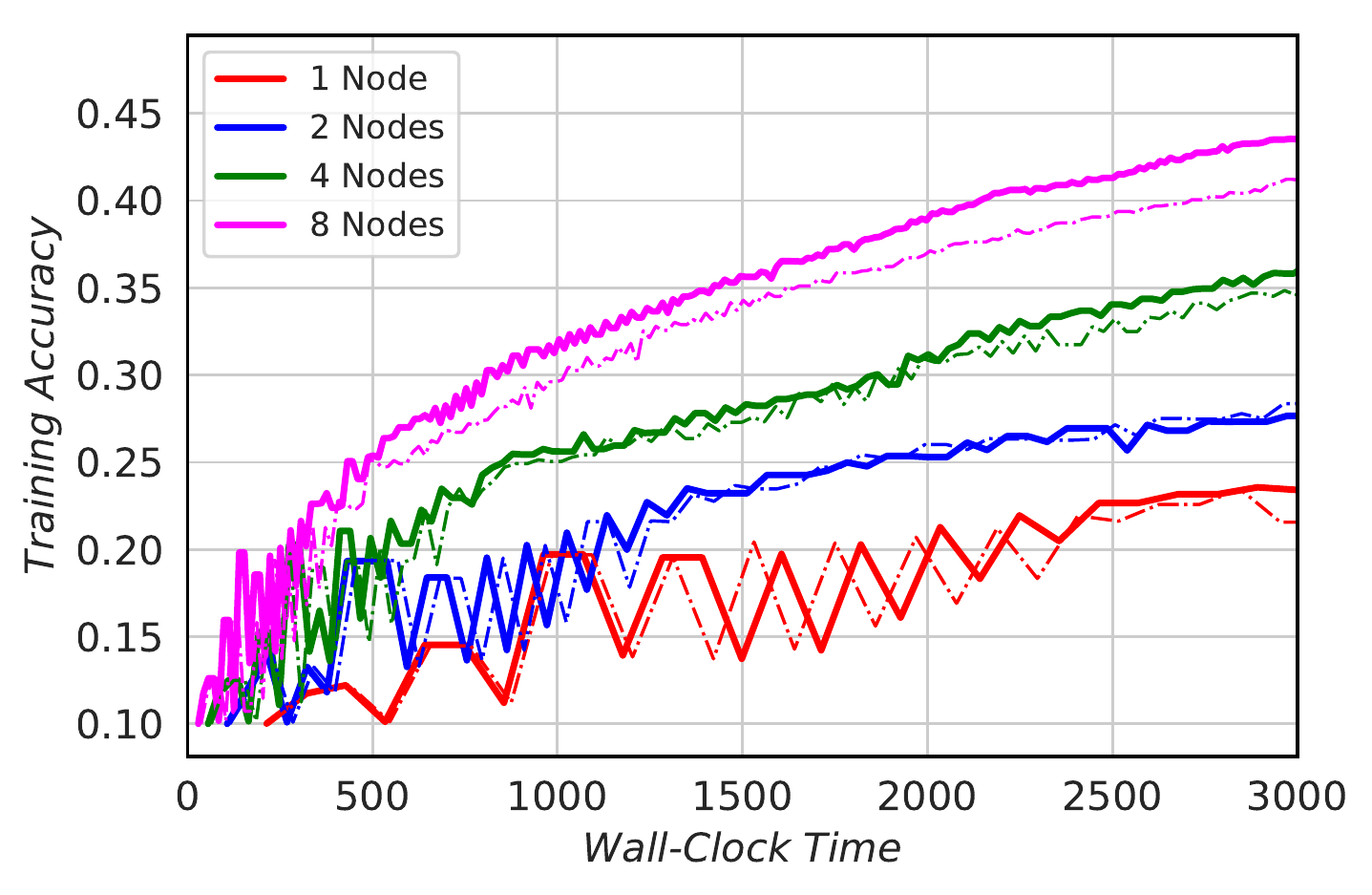}
			\includegraphics[width=0.44\textwidth]{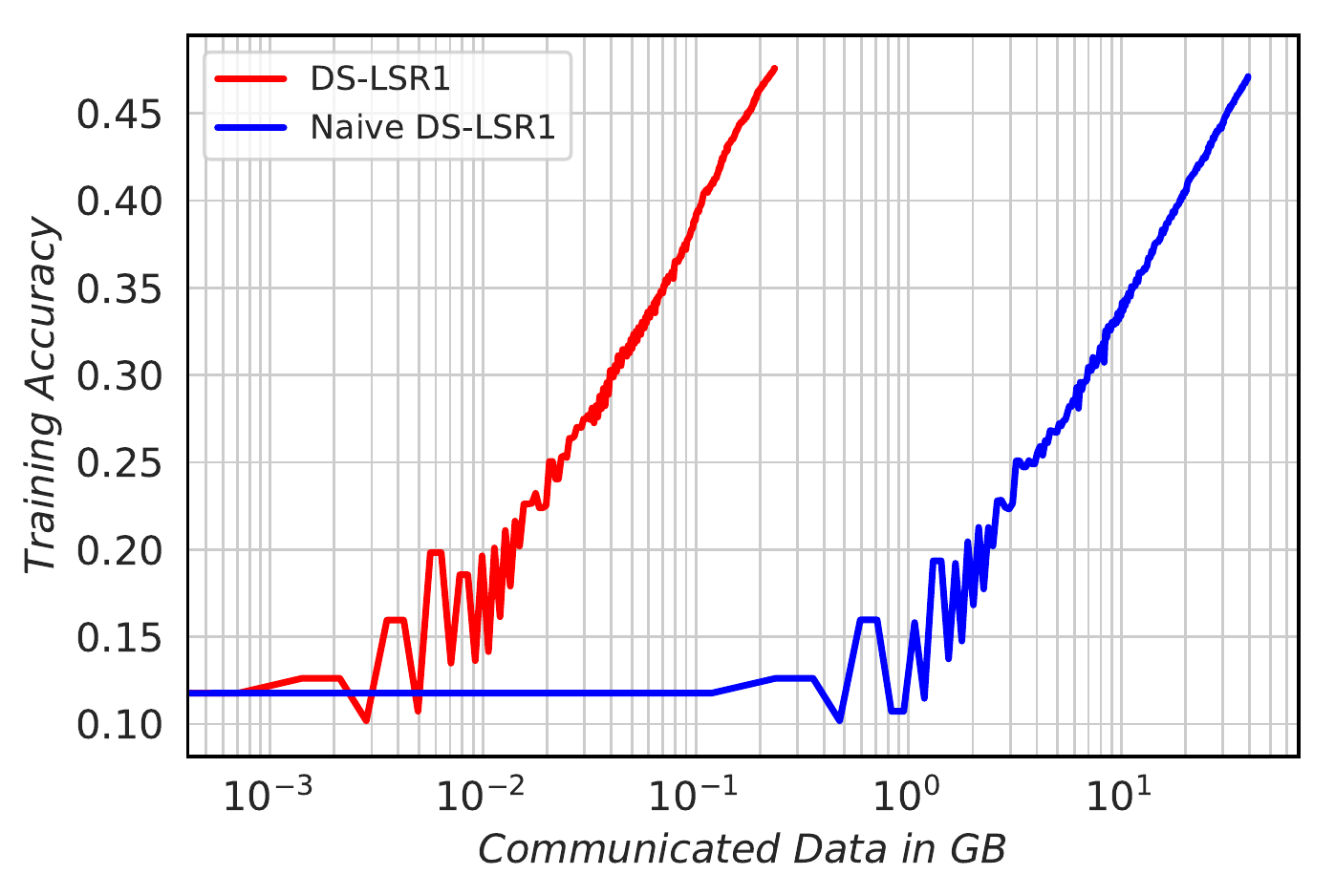}
			
		    \includegraphics[width=0.44\textwidth]{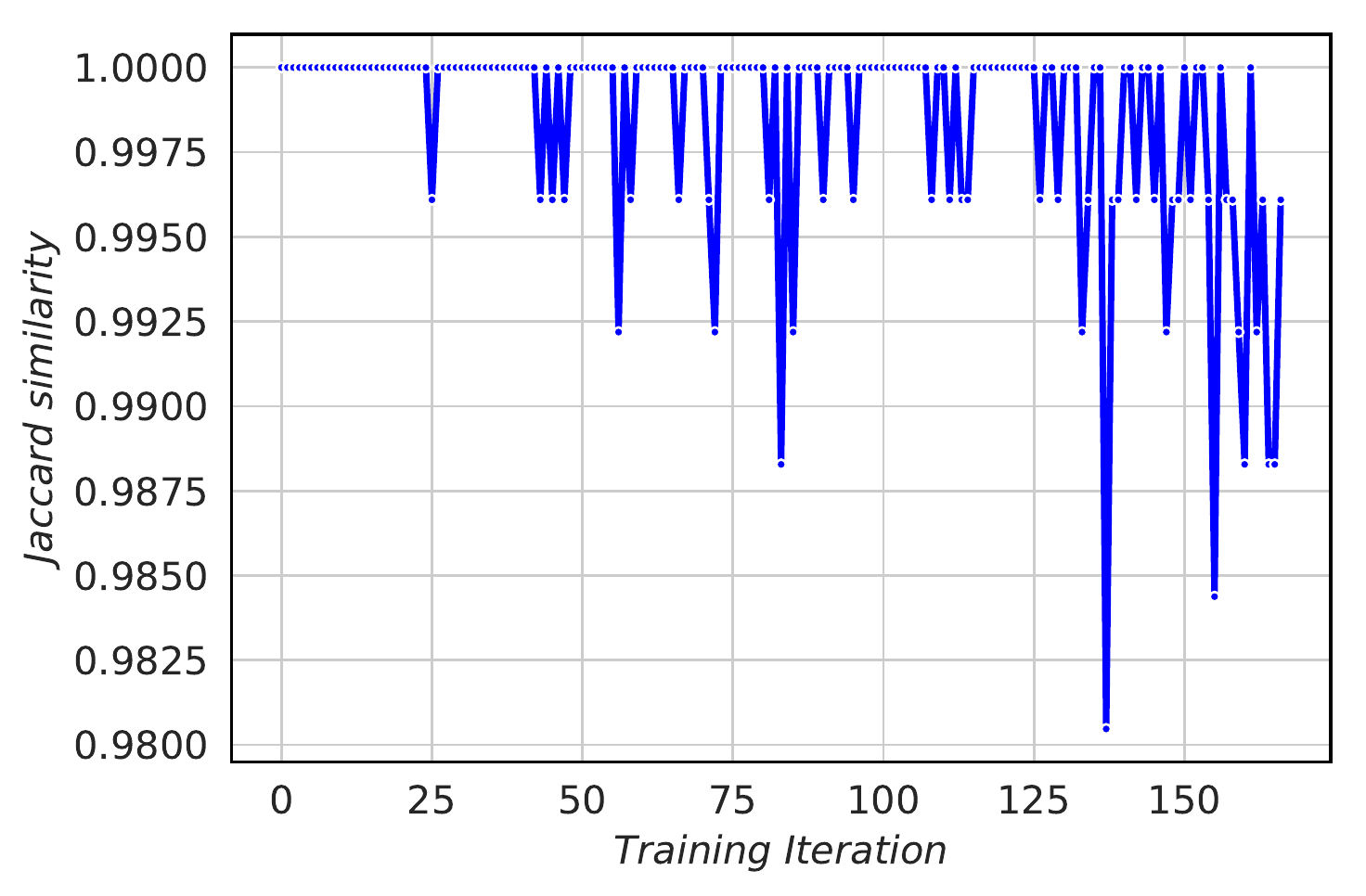}
		    \vspace{-5 pt}
		\caption{Performance of DS-LSR1 on CIFAR10 dataset with different number of nodes.}
		\label{fig:perf}
\end{figure}

Finally, we show that the curvature pairs chosen by our approach are almost identical to those chosen by the naive approach even though we use an approximation (via sketching) when checking the SR1 condition. Figure \ref{fig:perf} (right plot), shows the Jaccard similarity for the sets of curvature pairs selected by the methods; the pairs are almost identical, with differences on a few iterations. \vspace{-10 pt}

\section{Final Remarks}
\label{sec:fin_rem}

This paper describes a scalable distributed implementation of the sampled LSR1 method which is communication-efficient, has favorable work-load balancing across nodes and that is matrix-free and inverse-free. The method leverages the compact representation of SR1 matrices and uses sketching techniques to drastically reduce the amount of data communicated at every iteration as compared to a naive distributed implementation. The DS-LSR1 method scales well in terms of both the dimension of the problem and the number of data points.

\subsubsection*{Acknowledgements} This work was partially supported by the U.S. National Science Foundation, under award numbers NSF:CCF:1618717, NSF:CMMI:1663256 and NSF:CCF:1740796, and XSEDE Startup grant IRI180020.

%
% ---- Bibliography ----
%
% BibTeX users should specify bibliography style 'splncs04'.
% References will then be sorted and formatted in the correct style.
%
{\small
\bibliographystyle{splncs04}
\bibliography{reference}
}

\appendix

\section{Theoretical Results and Proofs}
\label{app:theory}

In this section, we prove a theoretical result about the matrix $(M_k^{(j)})^{-1}$.

\begin{lemma}\label{lem:MinvRecursive}
The matrix $M_k^{(j+1)}$, for $j=0,\dots,m-1$, has the form:
\begin{equation}\label{eq:Mrec}
M_k^{(j+1)}=
\left[
\begin{array}{c:c}
M_k^{(j)} & u \\
\hdashline
v^T & c
\end{array}
\right],
\end{equation}
where $v^T = s_{k,j+1}^TY_{k,1:l}$ and $l \leq j$, $u =v$ and  $c =s_{k,j+1}^T y_{k,j+1}$, and is nonsingular. Moreover, its inverse can be calculated as following:
\begin{equation}\label{eq:MinvRec}
(M_k^{(j+1)})^{-1}=
\left[
\begin{array}{c:c}
(M_k^{(j)})^{-1}+\zeta(M_k^{(j)})^{-1}uv^T(M_k^{(j)})^{-1} & -\zeta(M_k^{(j)})^{-1} u \\
\hdashline
- \zeta v^T(M_k^{(j)})^{-1} & \zeta
\end{array}
\right]
\end{equation}
where $\zeta = \dfrac{1}{c - v^T(M_k^{(j)})^{-1}u}$.
\end{lemma}
\begin{proof}
It is trivial to show that $M_k^{(j+1)}$ shown in \eqref{eq:Mrec} is equivalent to the corresponding matrix in \eqref{eq:hessFreeD}. Moreover, the second part of the lemma follows immediately from the fact that $M_k^{(i+1)}$ is itself non-singular and symmetric as shown in \cite{byrd1994}.
Lets consider the following matrix $M_k^{(i+1)}$:
\begin{equation}\label{eq:Mrec2}
M_k^{(j+1)}=
\left[
\begin{array}{c:c}
M_k^{(j)} & u \\
\hdashline
v^T & c
\end{array}
\right]
\end{equation}
We know that $M_k^{(i)}$ is invertible, and in the following by simple linear algebra, we calculate the inverse of $M_k^{(i+1)}$: 
\begin{align*}\label{eq:inv}
\left[
\begin{array}{c:c|c:c}
M_k^{(j)} & u & I & 0 \\
\hdashline
v^T & c &0 & 1
\end{array}
\right] &	\Rightarrow \left[
\begin{array}{c:c|c:c}
I & (M_k^{(j)})^{-1}u & (M_k^{(j)})^{-1} & 0 \\
\hdashline
v^T & c &0 & 1
\end{array}
\right]\\
&\Rightarrow
\left[
\begin{array}{c:c|c:c}
I & (M_k^{(j)})^{-1}u & (M_k^{(j)})^{-1} & 0 \\
\hdashline
0 & c - v^T(M_k^{(j)})^{-1}u &-v^T(M_k^{(j)})^{-1} & 1
\end{array}
\right]\\
&\Rightarrow
\left[
\begin{array}{c:c|c:c}
I & (M_k^{(j)})^{-1}u & (M_k^{(j)})^{-1} & 0 \\
\hdashline
0 & 1 &\dfrac{-v^T(M_k^{(j)})^{-1}}{c - v^T(M_k^{(j)})^{-1}u} & \dfrac{1}{c - v^T(M_k^{(j)})^{-1}u}
\end{array}
\right]\\
&\Rightarrow 
\left[
\begin{array}{c:c|c:c}
I & 0 & (M_k^{(j)})^{-1}+ \dfrac{(M_k^{(j)})^{-1}uv^T(M_k^{(j)})^{-1}}{c - v^T(M_k^{(j)})^{-1}u} & \dfrac{-(M_k^{(j)})^{-1}u}{c - v^T(M_k^{(j)})^{-1}u} \\
\hdashline
0 & 1 &\dfrac{-v^T(M_k^{(j)})^{-1}}{c - v^T(M_k^{(j)})^{-1}u} & \dfrac{1}{c - v^T(M_k^{(j)})^{-1}u}
\end{array}
\right]\\
&\Rightarrow 
\left[
\begin{array}{c:c|c:c}
I & 0 & (M_k^{(j)})^{-1}+ \zeta(M_k^{(j)})^{-1}uv^T(M_k^{(j)})^{-1} & -\zeta(M_k^{(j)})^{-1}u \\
\hdashline
0 & 1 &-\zeta v^T(M_k^{(j)})^{-1} & \zeta
\end{array}
\right]
\end{align*}
The last line is by putting $\zeta = \dfrac{1}{c - v^T(M_k^{(j)})^{-1}u}$.

\end{proof}

Lemma \ref{lem:MinvRecursive} describes a recursive  method for computing $(M_k^{(j)})^{-1} \in \mathbb{R}^{j \times j}$, for $j=1,\dots,m$. Specifically, one can calculate $(M_k^{(j+1)})^{-1}$ using $(M_k^{(j)})^{-1}$. We should note, that the first matrix $(M_k^{(1)})^{-1}$ is simply a number. Overall, this procedure allows us to compute $(M_k^{(j)})^{-1}$ without explicitly computing an inverse.

%%%%%%%%%%%%%%%%%%%%%%%%%%%%%%%%%%%%%%%%%%
%%%%%%%%%%%%%%%%%%%%%%%%%%%%%%%%%%%%%%%%%%
\section{Additional Algorithm Details}
\label{app:algs}
In this section, we present additional details about the S-LSR1 and DS-LSR1 algorithms discussed in the Sections \ref{sec:slsr1} and \ref{sec:ds-lsr1}. 

\subsection{CG Steihaug Algorithm - Serial}
\label{sec:CG_serial}
In this section, we describe CG-Steihaug Algorithm \cite{nocedal_book} which is used for computing the search direction $p_k$. 
\begin{algorithm}[H]
\small 
\caption{CG-Steihaug (Serial)}
  \label{alg:CG-Steihaug(Serial)}
 {\bf Input:} $\epsilon$ (termination tolerance), $\nabla F(w_k)$ (current gradient).

  \begin{algorithmic}[1]
    \State Set $z_0 =0$, $r_0= \nabla F(w_k)$, $d_0 = -r_0$
   \If{ $\| r_0\| < \epsilon$}
    \State {\bf return} $p_k = z_0 = 0$
  \EndIf
  %\State Compute $\nabla F(w_0)$ 
  \For {$j=0,1,2,...$}
     \If{ $d_j^T\textcolor{magenta}{B_kd_j} \leq 0 $}
        \State Find $\tau \geq 0$ such that $p_k = z_j + \tau d_j$ minimizes $m_k(p_k)$ and satisfies $\|p_k\| = \Delta_k$
        \State {\bf return} $p_k $
    \EndIf
    \State Set $\alpha_j = \dfrac{r_j^Tr_j}{d_j^TB_kd_j}$ and $z_{j+1}= z_j + \alpha_jd_j$
     \If{ $\| z_{j+1}\| \geq \Delta_k$}
        \State Find $\tau\geq 0$ such that $p_k = z_j + \tau d_j$ and satisfies $\|p_k\| = \Delta_k$
        \State {\bf return} $p_k $
    \EndIf
    \State Set $r_{j+1}= r_j + \alpha_jB_kd_j$
    \If{ $\| r_{j+1}\| < \epsilon_k$}
        \State {\bf return} $p_k = z_{j+1}$
    \EndIf 
    \State Set $\beta_{j+1} =\dfrac{r_{j+1}^Tr_{j+1}}{ r_{j}^Tr_{j}} $ and $d_{j+1} =-r_{j+1} + \beta_{j+1}d_j $
\EndFor
  \end{algorithmic}
\end{algorithm}

\clearpage

\subsection{CG Steihaug Algorithm - Distributed}
\label{sec:CG_dist}
In this section, we describe a distributed variant of CG Steihaug algorithm that is used as a subroutine of the DS-LSR1 method. The manner in which Hessian vector products are computed was discussed in Section \ref{sec:ds-lsr1}. 
\begin{algorithm}[H]
{\small 
\caption{CG-Steihaug (Distributed)}
  \label{alg:CG-Steihaug(Distributed)}
 {\bf Input:} $\epsilon$ (termination tolerance), $\nabla F(w_k)$ (current gradient), $M_k^{-1}$. 

\textbf{Master Node:} \hfill \textbf{Worker Nodes (\pmb{$i=1,2,\dots,\mathcal{K}$}):} 
  \begin{algorithmic}[1]
  \State Set $z_0 =0$, $r_0= \nabla F(w_k)$, $d_0 = -r_0$
   \If{ $\| r_0\| < \epsilon_k$}
    \State {\bf return} $p_k = z_0 = 0$
  \EndIf
  %\State Compute $\nabla F(w_0)$ 
  \For {$j=0,1,2,...$}
  \State {\color{green!40!black}{\it \textbf{Broadcast:}}} $d_j$, $M_k^{-1}$ \hfill \textcolor{NavyBlue}{$\pmb{\longrightarrow}$} \hfill Compute  $M_k^{-1}(Y_{k}^{i})^Td_j$
  \State {\color{orange!50!black}{\it \textbf{Reduce:}}} $M_k^{-1}(Y_{k}^{i})^Td_j$ to $ M_k^{-1}Y_k^Td_j$ \hfill \textcolor{NavyBlue}{$\pmb{\longleftarrow}$} \hfill $\ \ \ $ 
   \State {\color{green!40!black}{\it \textbf{Broadcast:}}} $ M_k^{-1}Y_k^Td_j$ \hfill \textcolor{NavyBlue}{$\pmb{\longrightarrow}$} \hfill Compute  $Y_k^{i} M_k^{-1}Y_k^Td_j$
  \State {\color{orange!50!black}{\it \textbf{Reduce:}}} $Y_k^iM_k^{-1}Y_k^Td_j$ to $B_kd_j = Y_k M_k^{-1}Y_k^Td_j$ \hfill \textcolor{NavyBlue}{$\pmb{\longleftarrow}$} \hfill $\ \ \ $
     \If{ $d_j^TB_kd_j \leq 0 $}
        \State Find $\tau \geq 0$ such that $p_k = z_j + \tau d_j$ minimizes $m_k(p_k)$ and satisfies $\|p_k\| = \Delta_k$
        \State {\bf return} $p_k $
    \EndIf
    \State Set $\alpha_j = \dfrac{r_j^Tr_j}{d_j^TB_kd_j}$ and $z_{j+1}= z_j + \alpha_jd_j$
     \If{ $\| z_{j+1}\| \geq \Delta_k$}
        \State Find $\tau\geq 0$ such that $p_k = z_j + \tau d_j$ and satisfies $\|p_k\| = \Delta_k$
        \State {\bf return} $p_k $
    \EndIf
    \State Set $r_{j+1}= r_j + \alpha_jB_kd_j$
    \If{ $\| r_{j+1}\| < \epsilon_k$}
        \State {\bf return} $p_k = z_{j+1}$
    \EndIf 
    \State Set $\beta_{j+1} =\dfrac{r_{j+1}^Tr_{j+1}}{ r_{j}^Tr_{j}} $ and $d_{j+1} =-r_{j+1} + \beta_{j+1}d_j $
\EndFor
  \end{algorithmic}
  }
\end{algorithm}

\newpage

\subsection{Trust-Region Management Subroutine}
\label{sec:tr_alg}
In this section, we present the Trust-Region management subroutine $\Delta_{k+1} = \texttt{adjustTR}(\Delta_{k},\rho_k)$. See \cite{nocedal_book} for further details.

\begin{center}
\begin{minipage}{.75\linewidth}
\begin{algorithm}[H]
\caption{$\Delta_{k+1} = \texttt{adjustTR}(\Delta_{k},\rho_k,\eta_2,\eta_3,\gamma_1,\zeta_1,\zeta_2)$: Trust-Region management subroutine}
  \label{alg:tr_mgmt}
 {\bf Input:} $\Delta_k$ (current trust region radius), $0 \leq \eta_3 < \eta_2 < 1$, $\gamma_1 \in (0,1)$, $\zeta_1 > 1$, $\zeta_2 \in (0,1)$ (trust region parameters). 
  \begin{algorithmic}[1]
    \If {$\rho_k > \eta_2$}
        \If {$\|p_k\| \leq \gamma_1 \Delta_k$}
            \State Set $\Delta_{k+1} = \Delta_k$
        \Else
            \State Set $\Delta_{k+1} = \zeta_1 \Delta_k$
        \EndIf
    \ElsIf{$\eta_3 \leq \rho_k \leq \eta_2$}       \State Set         $\Delta_{k+1} =     \Delta_k$
    \Else
        \State $\Delta_{k+1} = \zeta_2 \Delta_k$
    \EndIf  
  \end{algorithmic}
\end{algorithm}
\end{minipage}
\end{center}

\subsection{Load Balancing}
In distributed algorithms, it is very important to have work-load balancing across nodes. In order for an algorithm to be scalable, every machine (worker) should have similar amount of assigned computation, and each machine should be equally busy. According to Amdahl's law \cite{rodgers1985improvements} if the parallel/distributed algorithm runs $t$ portion of  time only on one of the machines (e.g., the master node), the theoretical speedup (SU) is limited to at most 
\begin{equation}\label{adamhl}
SU \leq \dfrac{1}{t + \dfrac{(1-t)}{\mathcal{K}}}.  
\end{equation}

As is clear from Tables \ref{tbl:comm} and \ref{tbl:comp}, the DS-LSR1 method makes each machine almost equally busy, and as a result DS-LSR1 has a near linear speedup. On the other hand, in the naive DS-LSR1 approach the master node is significantly busier than the remainder of the nodes, and thus by Adamhl's law, the speedup will not be linear and is bounded above by \eqref{adamhl}. 

\subsection{Communication and Computation Details}
\label{sec:quant_details}

In this section, we present details about the quantities that are communicated and computed at every iteration of the distributed S-LSR1 methods. All the quantities below are in Tables \ref{tbl:comm} and \ref{tbl:comp}.

\begin{table}[H]
\caption{Details of quantities communicated and computed.}
\label{tbl:quant_details}
\centering
\begin{small}
\begin{tabular}{@{}ll@{}}\toprule
\textbf{Variable} & \textbf{Dimension} \\ \midrule
 $w_k$&  $d \times 1$
\\ \hdashline  
 $F(w_k), F_i(w_k)$ &  $1$
 \\ \hdashline  
 $\nabla F(w_k), \nabla F_i(w_k)$ &  $d \times 1$
\\ \hdashline
$p_k$ &  $d \times 1$
\\ \hdashline
 $S_k,S_{k,i}$&  $d \times m$
\\ \hdashline  
$Y_k,Y_{k,i}$&  $d \times m$ 
\\\hdashline  
$S_k^TY_{k,i},S_{k,i}^TY_{k,i}$&  $m \times m$ 
\\\hdashline 
 $M_k^{-1}$&  $m \times m$
 \\\hdashline  
 $B_kd$&  $d \times 1$
 \\\hdashline  
 $M_k^{-1}Y_{k,i}^Tp_k$ &  $m \times 1$
 \\\hdashline  
 $Y_{k,i}M_k^{-1}Y_k^Tp_k$ &  $d \times 1$
 \\\hdashline
 $M_k^{-1}$&  $m \times m$
\\  
\bottomrule 
\end{tabular}
\end{small}
\end{table}

\subsection{Floats Communicated per Iteration}
\label{sec:app_comm_graphs}

In this section, we should the number of floats communicated per iteration for DS-LSR1 and naive DS-LSR1 for different memory size and dimension.

\begin{figure}[H]
	\centering
 	\includegraphics[width=0.45\textwidth]{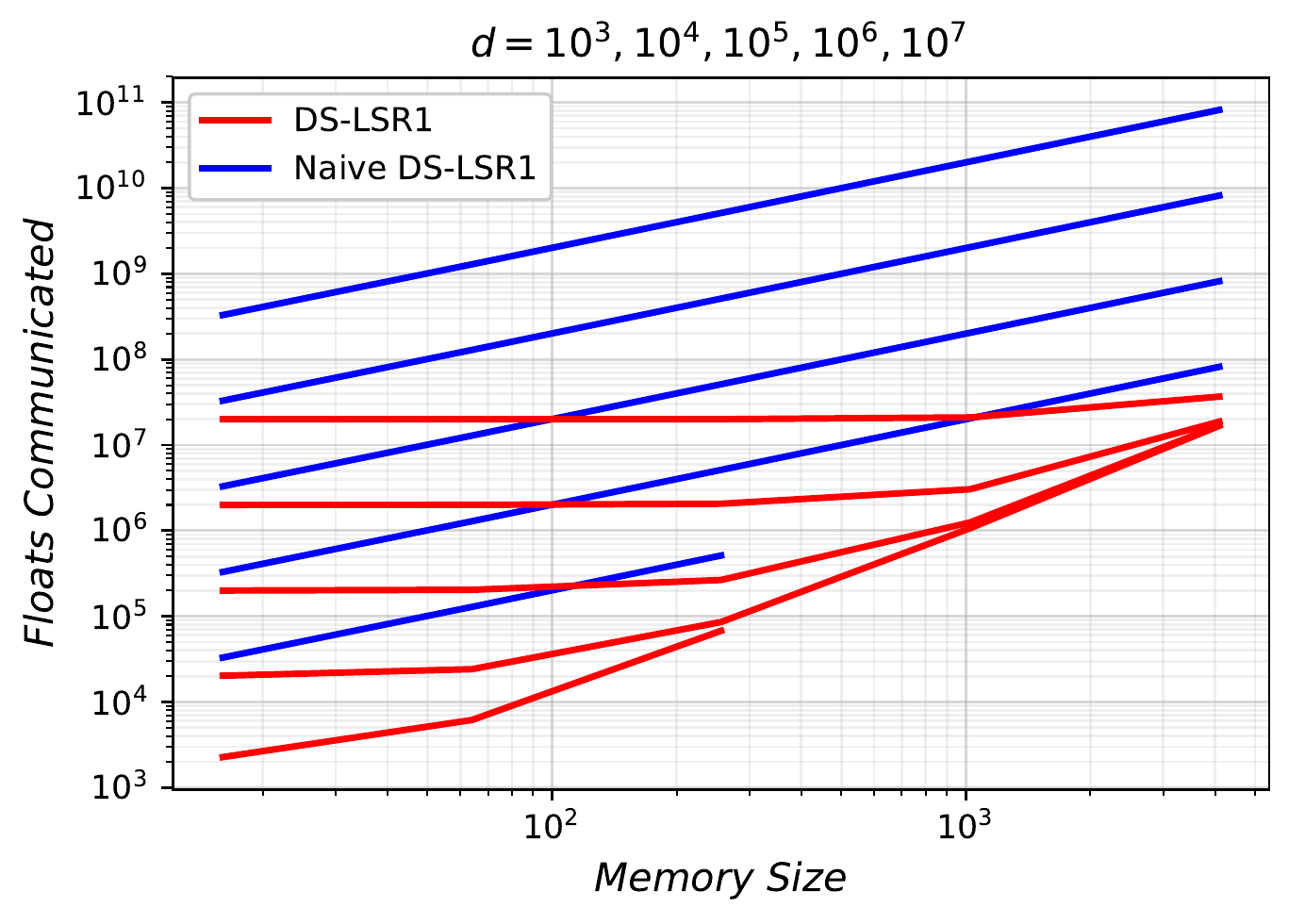}
 	\includegraphics[width=0.45\textwidth]{figs/floatNum_dim_fixedMmr.pdf}
 	\caption{Number of floats communicated ate every iteration for different dimension $d$ and memory size $m$. \label{fig:float_comm1}}
\end{figure}

\newpage

\newpage
%%%%%%%%%%%%%%%%%%%%%%%%%%%%%%%%%%%%%%%%%%
%%%%%%%%%%%%%%%%%%%%%%%%%%%%%%%%%%%%%%%%%%
\section{Additional Numerical Experiments and Experimental Details}
\label{app:NumericExper}

In this section, we present additional experiments and experimental details.

\subsection{Initial Hessian Approximation $B_k^{(0)}$}

In this section, we show additional results motivating the use of $B_k^{(0)}$. Figure \ref{SmalllvdNet1}, is identical to Figure \ref{SmalllvdNet}. Figure \ref{MedvdNet} shows similar results for a larger problem. See \cite{berahas2019sqn} for details about the problems.

\begin{figure}[]
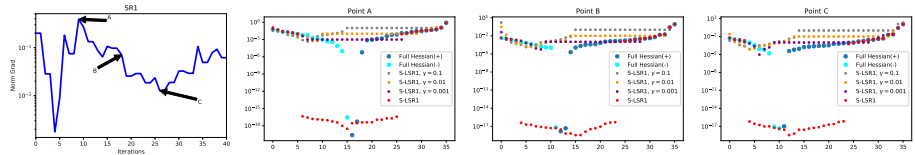

	\centering
	\includegraphics[trim=0.6cm 0.5cm 1.0cm 0.7cm,width=0.24\textwidth]{figs/normGradSR1_3pointsNewer.pdf}
 	\includegraphics[trim=0.6cm 0.7cm 1.0cm 0.7cm,width=0.24\textwidth]{figs/iter9nnSpectApr.pdf}
	\includegraphics[trim=0.6cm 0.7cm 1.0cm 0.7cm,width=0.24\textwidth]{figs/iter18nnSpectApr.pdf}
	\includegraphics[trim=0.6cm 0.7cm 1.0cm 0.7cm,width=0.24\textwidth]{figs/iter26nnSpectApr.pdf}
	\caption{Comparison of the eigenvalues of S-LSR1 for different $\gamma$ (@ A, B, C) for a \texttt{small} toy classification problem.\label{SmalllvdNet1}}
\end{figure}

\begin{figure}[]
	\centering
 
	\includegraphics[trim=0.6cm 0.5cm 1.0cm 0.7cm,width=0.24\textwidth]{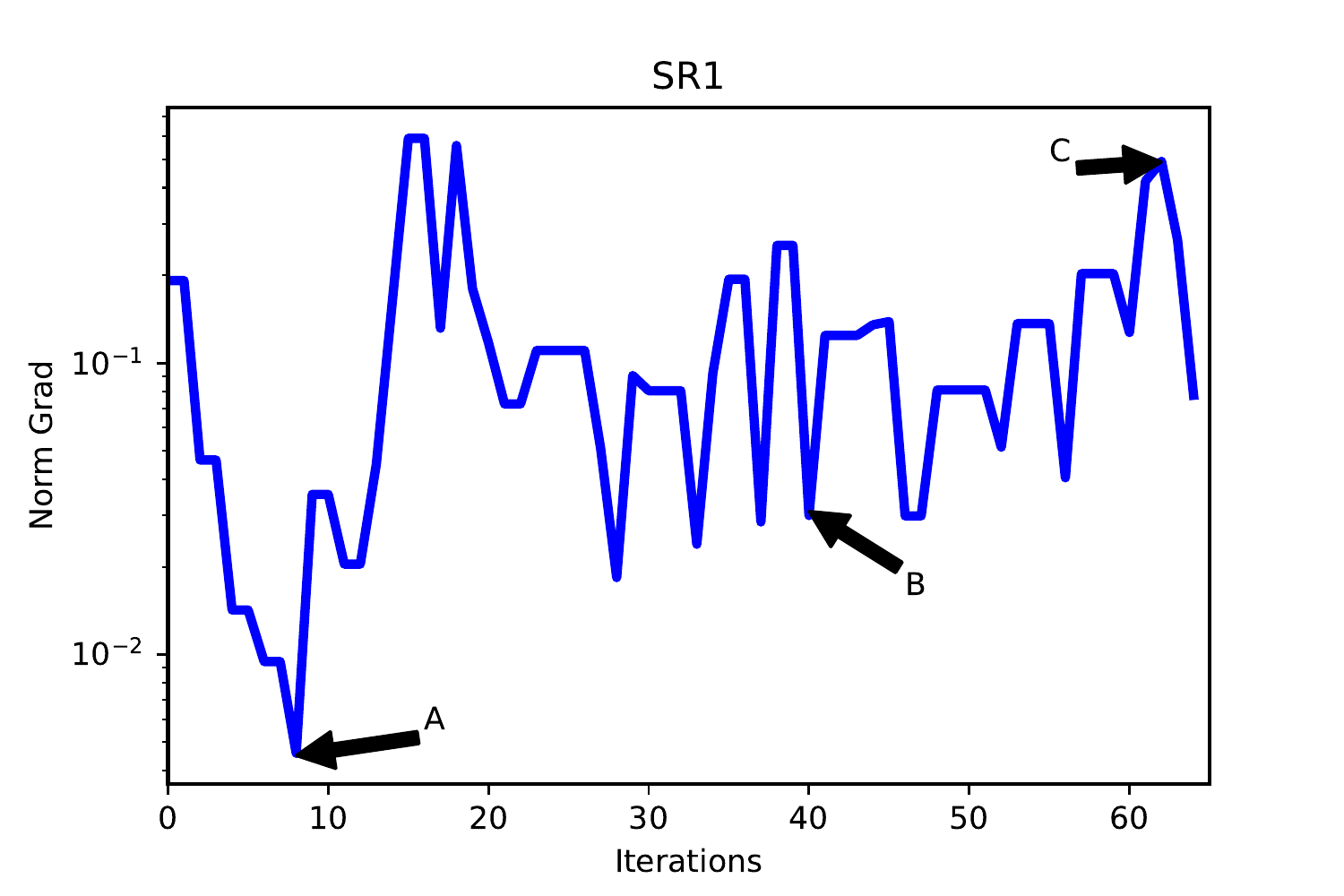}
 	\includegraphics[trim=0.6cm 0.7cm 1.0cm 0.7cm,width=0.24\textwidth]{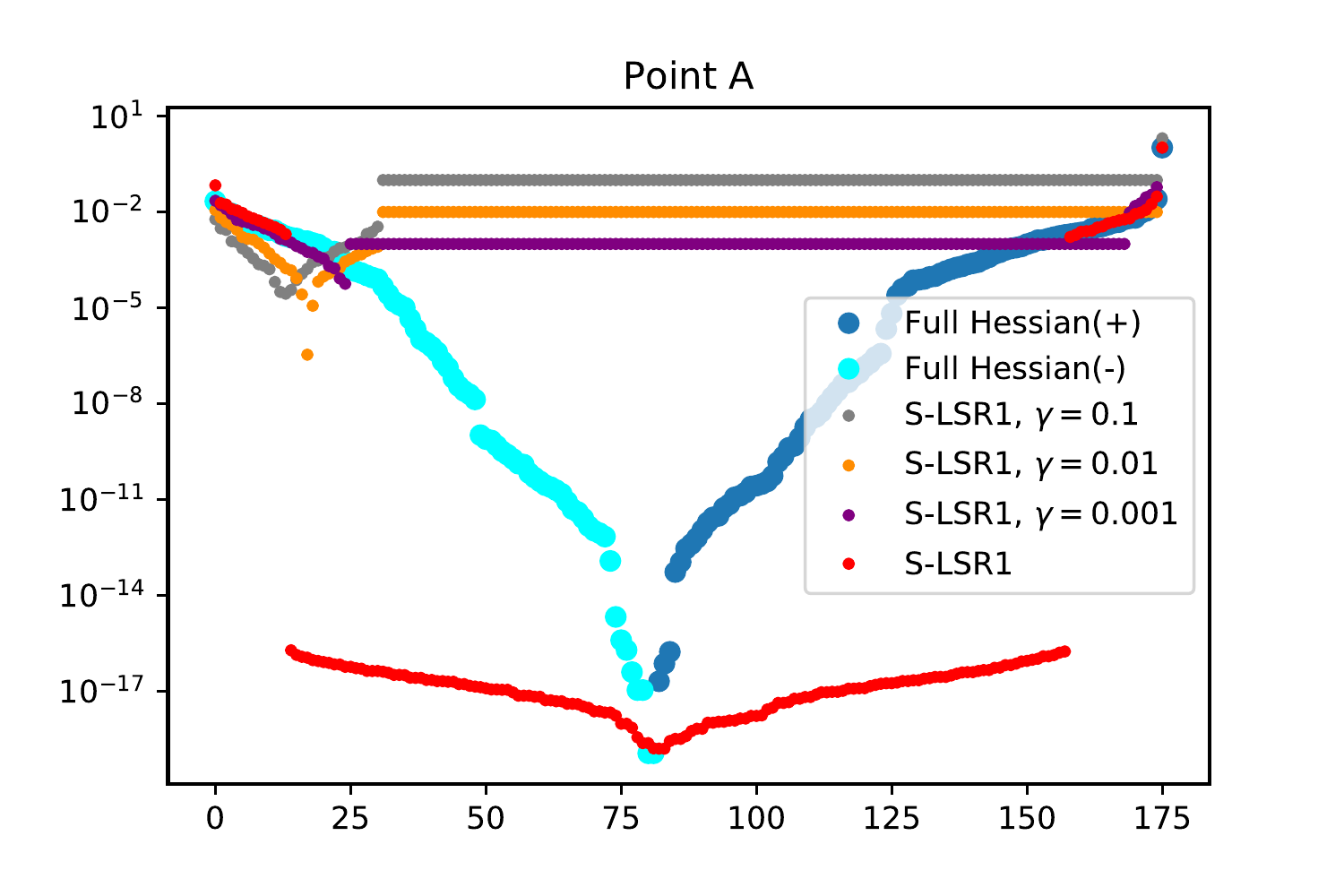}
	\includegraphics[trim=0.6cm 0.7cm 1.0cm 0.7cm,width=0.24\textwidth]{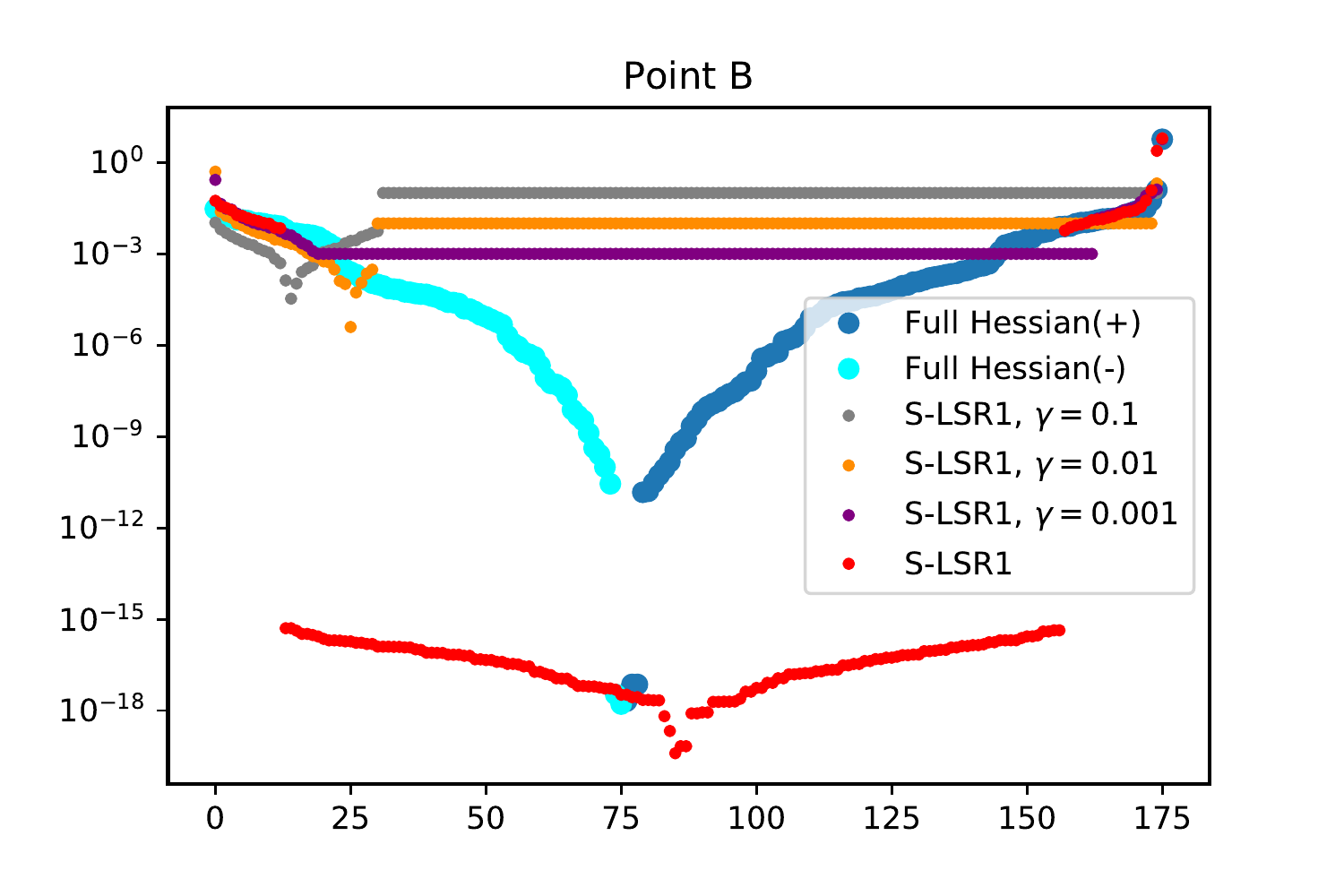}
	\includegraphics[trim=0.6cm 0.7cm 1.0cm 0.7cm,width=0.24\textwidth]{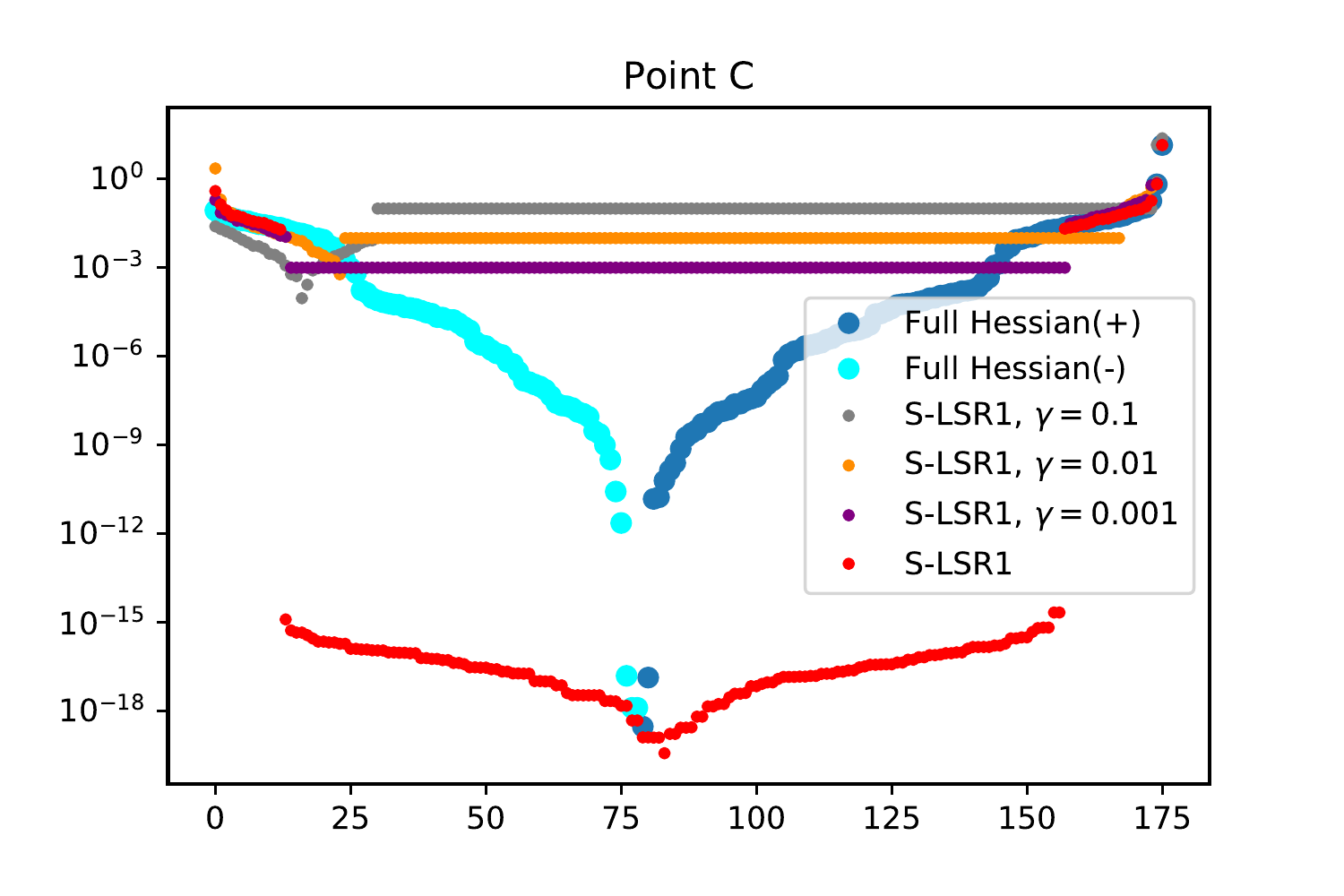}
	\caption{Comparison of the eigenvalues of S-LSR1 for different $\gamma$ (@ A, B, C) for a \texttt{medium} toy classification problem.\label{MedvdNet}}
\end{figure}

\newpage

\subsection{\texttt{Shallow} and \texttt{Deep} Network Details}
\label{app:shallow_deep}
In this section, we describe the networks used in the weak scaling experiments. For the problems corresponding to the Tables \ref{tbl:shallow} and \ref{tbl:deep}  we used ReLU activation functions and soft-max cross-entropy loss.

\begin{table}[H]
\caption{Details for \texttt{Shallow} Networks.}
\label{tbl:shallow}
\centering
\begin{footnotesize}
\begin{tabular}{@{}cccc@{}}\toprule
\textbf{Network} & \textbf{\shortstack{\# Hidden \\ Layers} } & \textbf{\shortstack{\# Nodes/ \\ Layer}}  & $\pmb{d}$\\ \midrule
1 & 1 & 1 & 805
\\ \hdashline  
2 & 1 & 10  & 7960
\\ \hdashline  
4 & 1 & 100  & 79510
\\ \hdashline  
3 & 1 & 1000  & 795010
\\
\bottomrule 
\end{tabular}
\end{footnotesize}
\end{table}

\begin{table}[H]
\caption{Details for \texttt{Deep} Networks.}
\label{tbl:deep}
\centering
\begin{footnotesize}
\begin{tabular}{@{}cccc@{}}\toprule
\textbf{Network} & \textbf{\shortstack{\# Hidden \\ Layers} } & \textbf{\shortstack{\# Nodes/ \\ Layer}}  & $\pmb{d}$\\ \midrule
1 & 7 & 2-2-2-2-2-2-2 & 817
\\ \hdashline  
2 & 7 & 10-10-10-10-10-10-10  & 8620
\\ \hdashline  
4 & 7 & 100-100-100-10-10-10-10  & 100150
\\ \hdashline  
3 & 7 & 1000-100-100-10-10-10-10  & 896650
\\
\bottomrule 
\end{tabular}
\end{footnotesize}
\end{table}

\newpage

\subsection{Weak Scaling}
\label{sec:app_strong}

In this section, we show the weak scaling properties of DS-LSR1 for two different networks, different batch sizes and different number of variables.

\begin{figure}[htbp]
  \centering
	\includegraphics[width=0.45\textwidth]{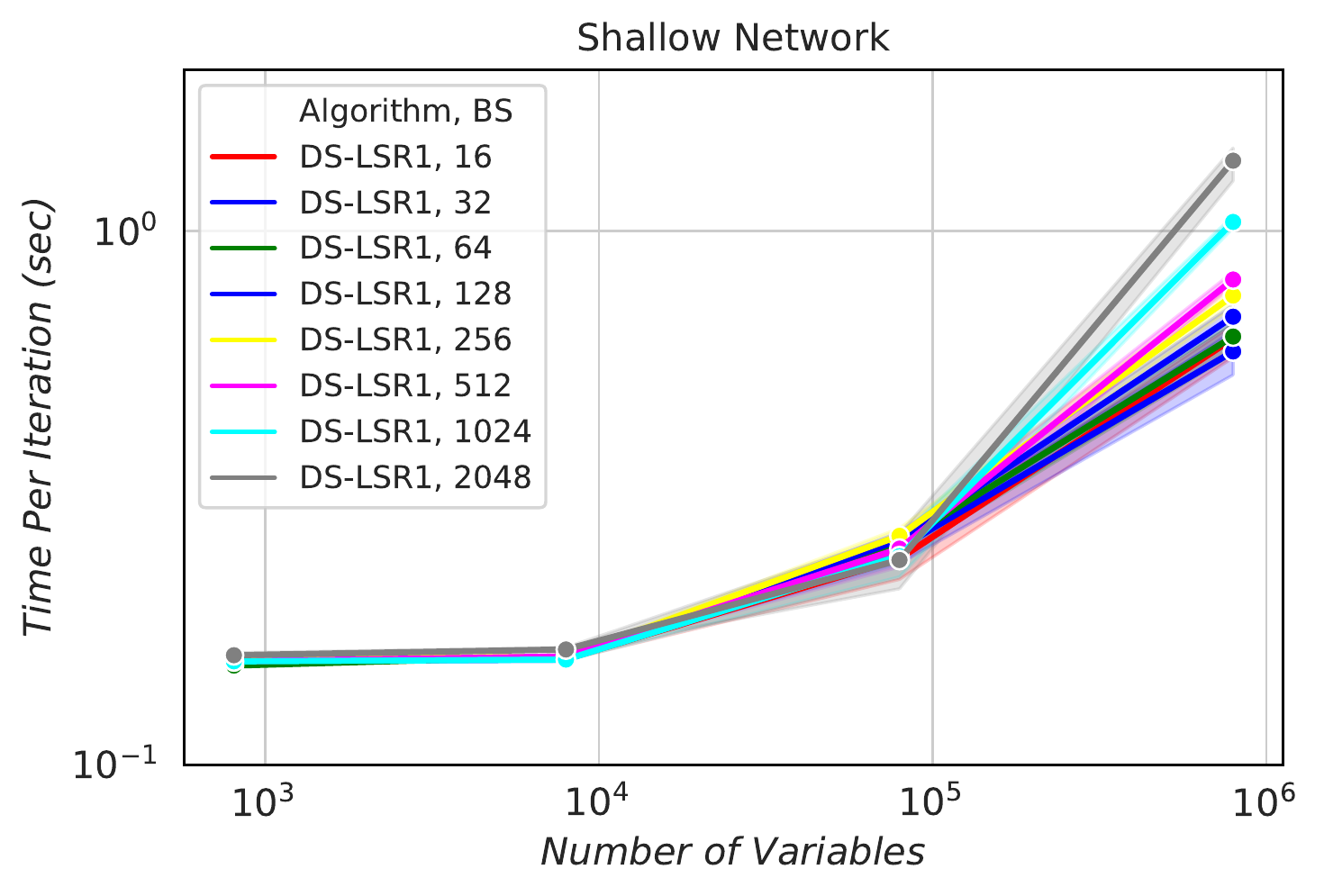}
 	\includegraphics[width=0.45\textwidth]{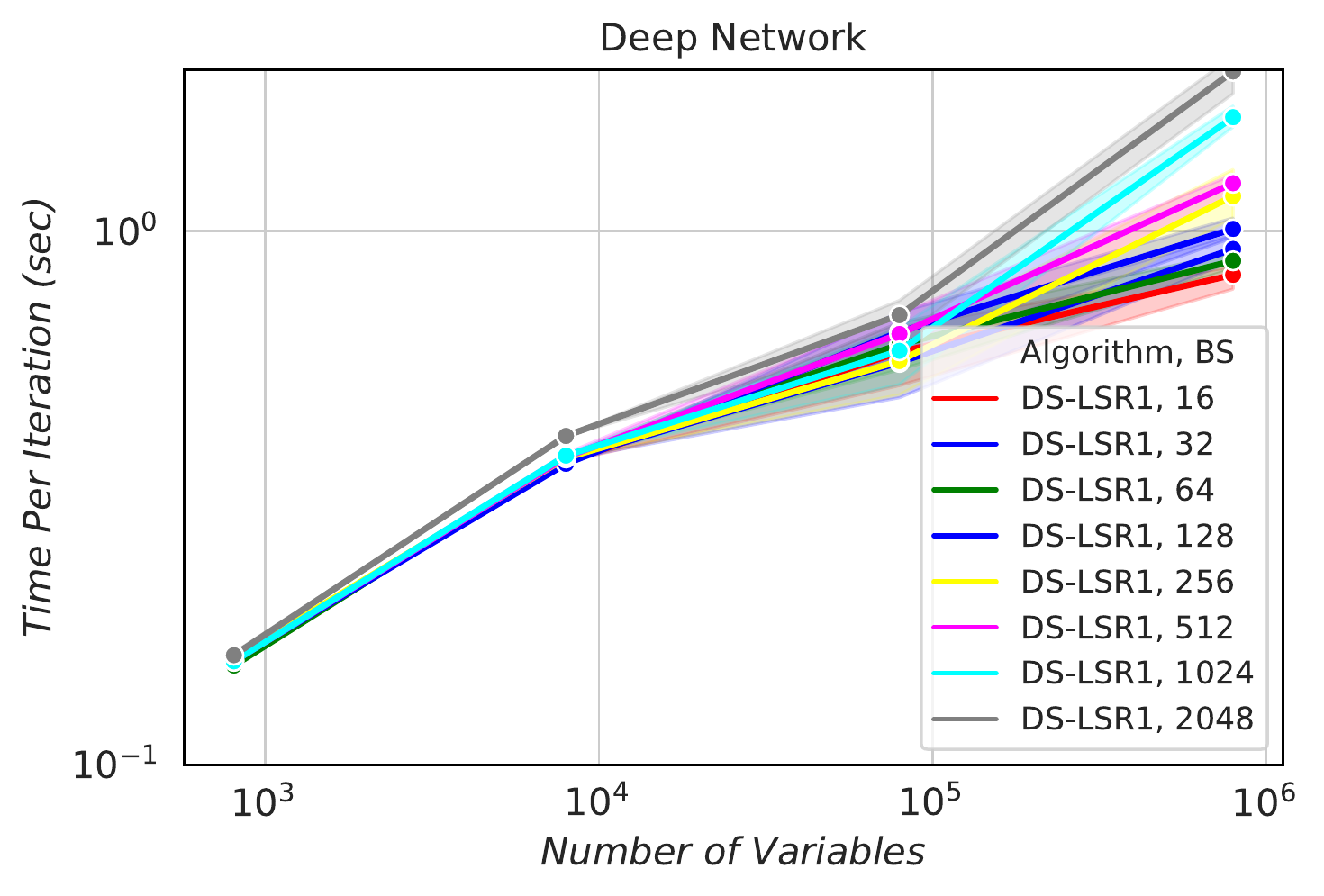}
	\caption{\textbf{Weak Scaling:} Time/iteration (sec) vs \# of variables; \texttt{Shallow} (left), \texttt{Deep} (right).}
\end{figure}

\subsection{Strong Scaling}
\label{sec:app_strong}

In this section, we show the strong scaling properties of DS-LSR1 and naive DS-LSR1 for different memory sizes. The problem details for these experiments were as follows: LeNet, CIFAR10, $d=62006$, \cite{lecun1998gradient}.

\begin{figure}[htbp]
  \centering
	\includegraphics[width=0.45\textwidth]{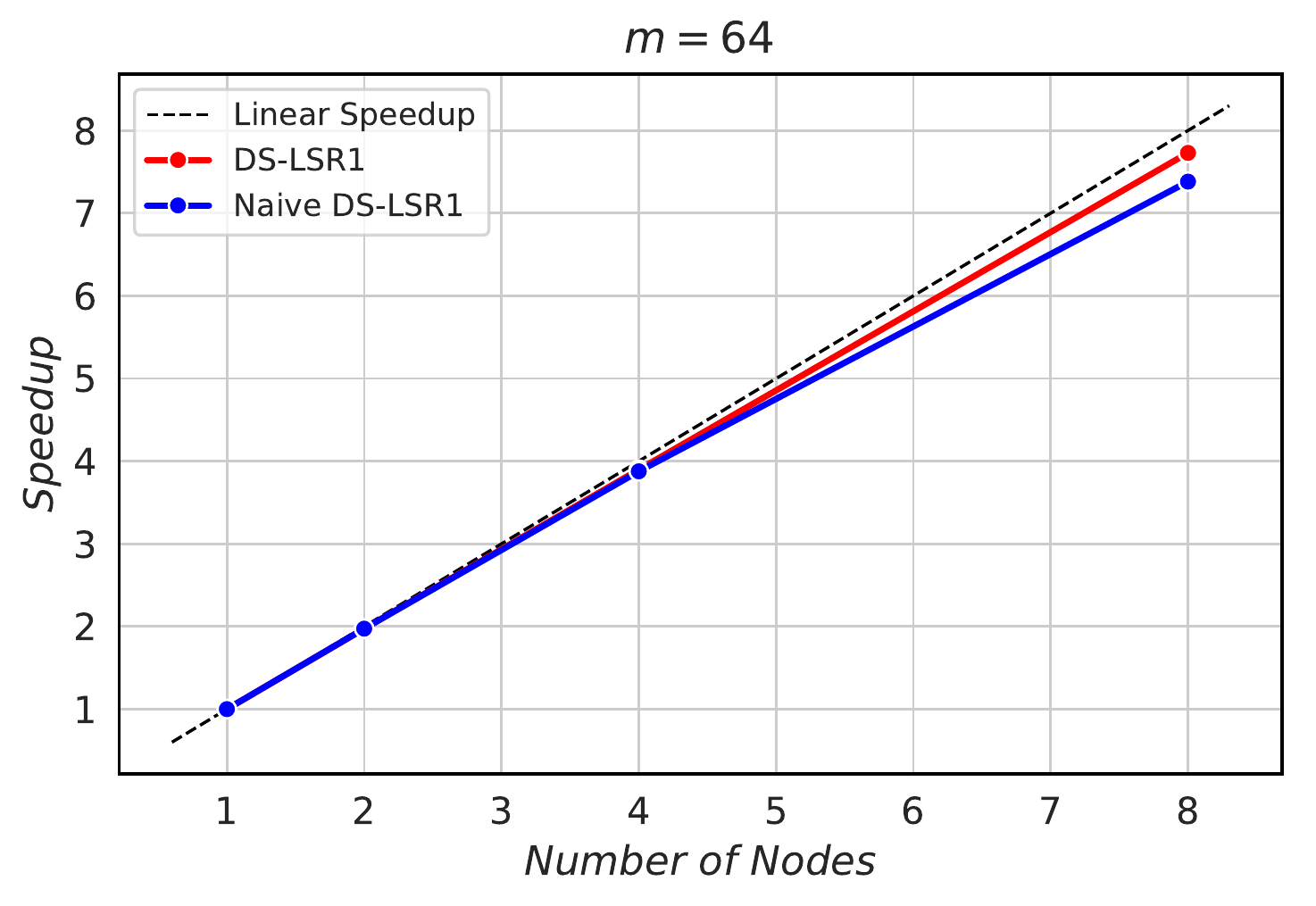}
 	\includegraphics[width=0.45\textwidth]{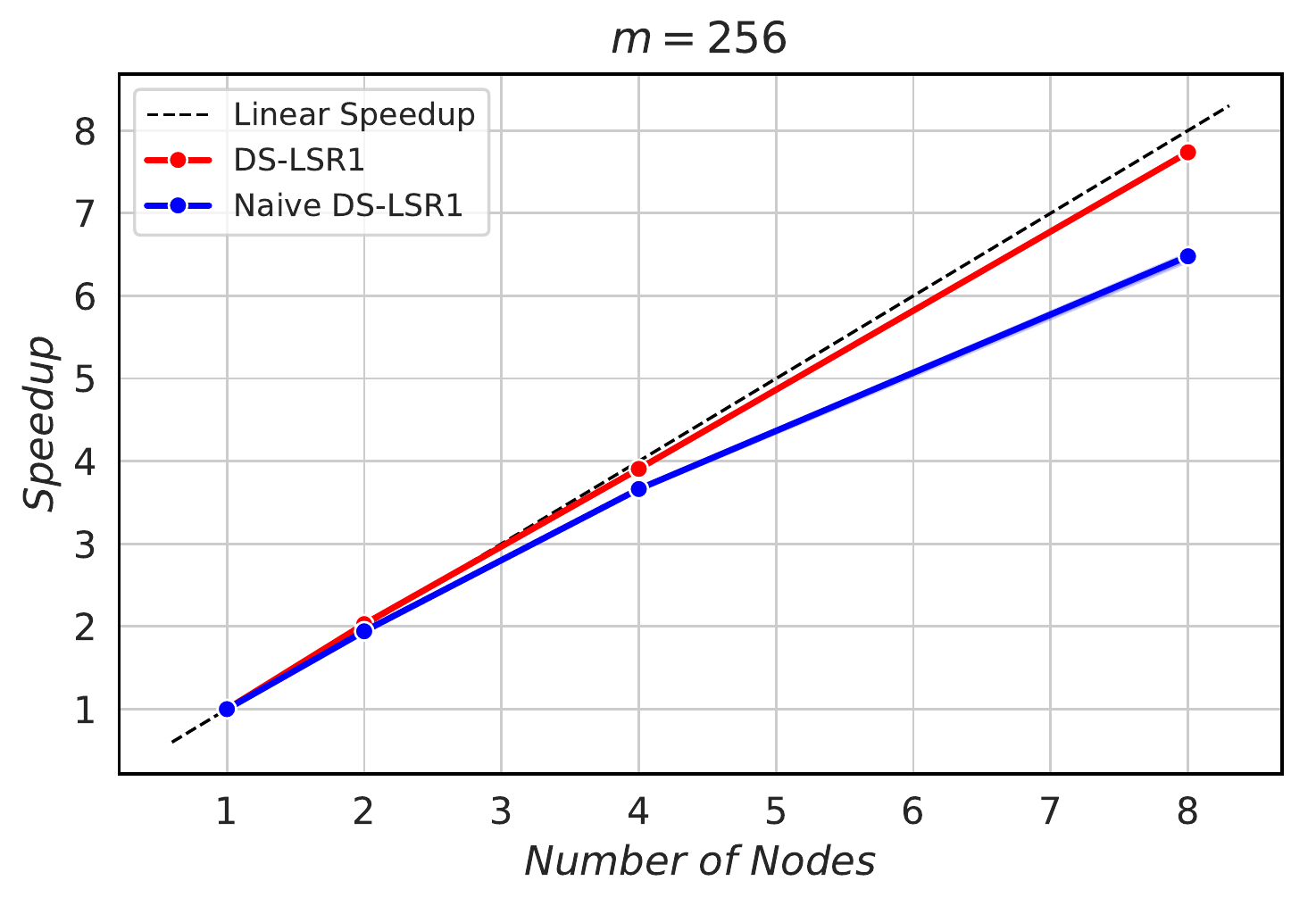}
	\caption{\textbf{Strong Scaling:} Relative speedup for different number of compute nodes and different memory levels: $64$ (left), $256$ (right). \label{fig:strong1}}
\end{figure}

\newpage
\subsection{Scaling  of  Different  Components  of  DS-LSR1}

In this section, we show the scaling properties of the different components of the DS-LSR1 method and compare with the naive distributed implementation. We deconstruct the main components of the DS-LSR1 method and illustrate the scaling with respect to memory. Specifically, we show the scaling for: $(1)$ reduce time/iteration; $(2)$ time/iteration; $(3)$ CG time/iteration; $(4)$ time to sample $S$, $Y$ pairs/iteration. For all these plots, we ran $10$ iterations and averaged the time, and also show the variability.

\begin{figure}[]
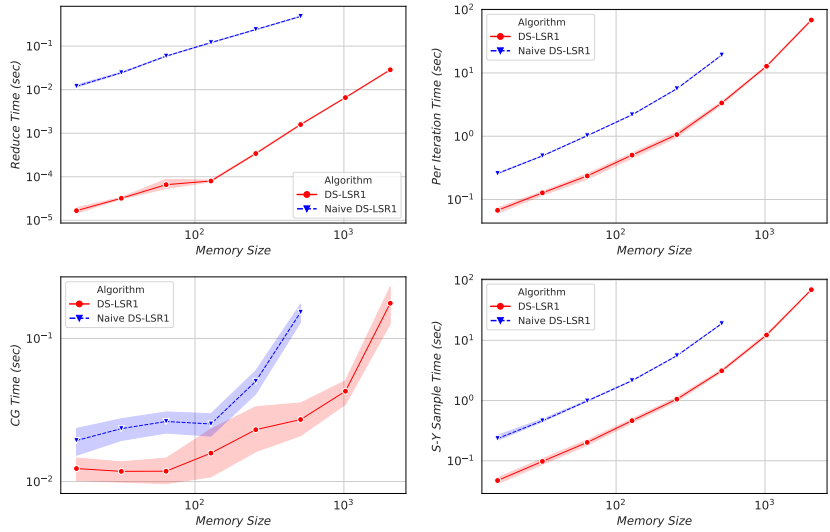

	\centering
  	\includegraphics[width=0.45\textwidth]{figs/both_mmr_error_band_Net1100_128_log.pdf}
 	\includegraphics[width=0.45\textwidth]{figs/both_mmr_error_band_Net1100_128_perIterTime_log.pdf} 
	
  	\includegraphics[width=0.45\textwidth]{figs/both_mmr_error_band_Net1100_128_CG_log.pdf}
  	\includegraphics[width=0.45\textwidth]{figs/both_mmr_error_band_Net1100_128_SY_log.pdf}
	\caption{Time (sec) for different components of DS-LSR1 with respect to memory.}
\end{figure}

\newpage
\subsection{Performance of DS-LSR1}
\label{sec:app_perf}

In this section, we show training and testing accuracy in terms of wall clock time and amount of data communicated (in GB).

\begin{figure}[htbp]
    \centering
		\includegraphics[width=0.45\textwidth]{figs/cifar_train_acc_mmr256_QR_large.pdf}
		\includegraphics[width=0.45\textwidth]{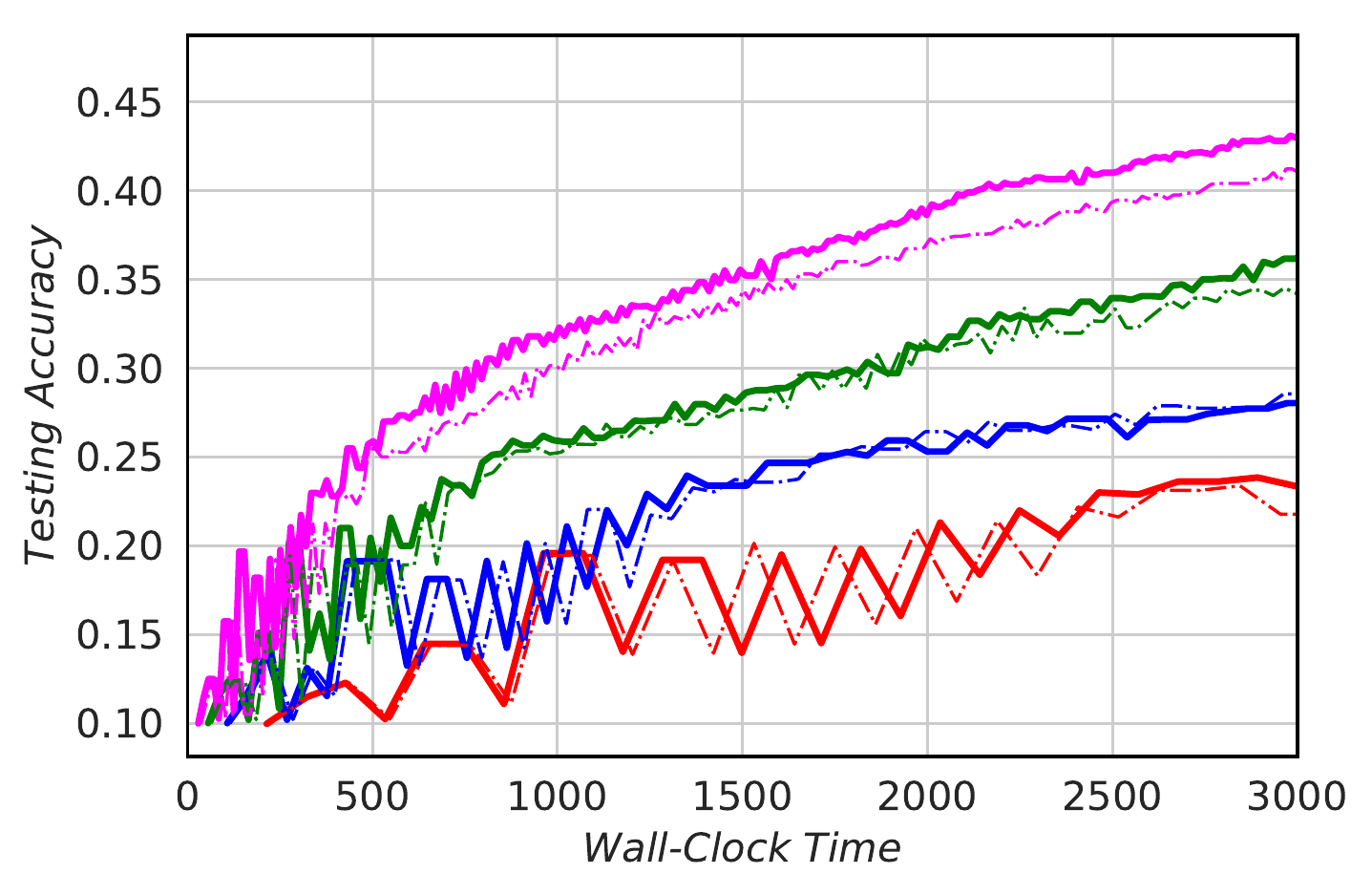}
			
		\includegraphics[width=0.45\textwidth]{figs/cifar_train_acc_mmr256_QR_float.pdf}
		\includegraphics[width=0.45\textwidth]{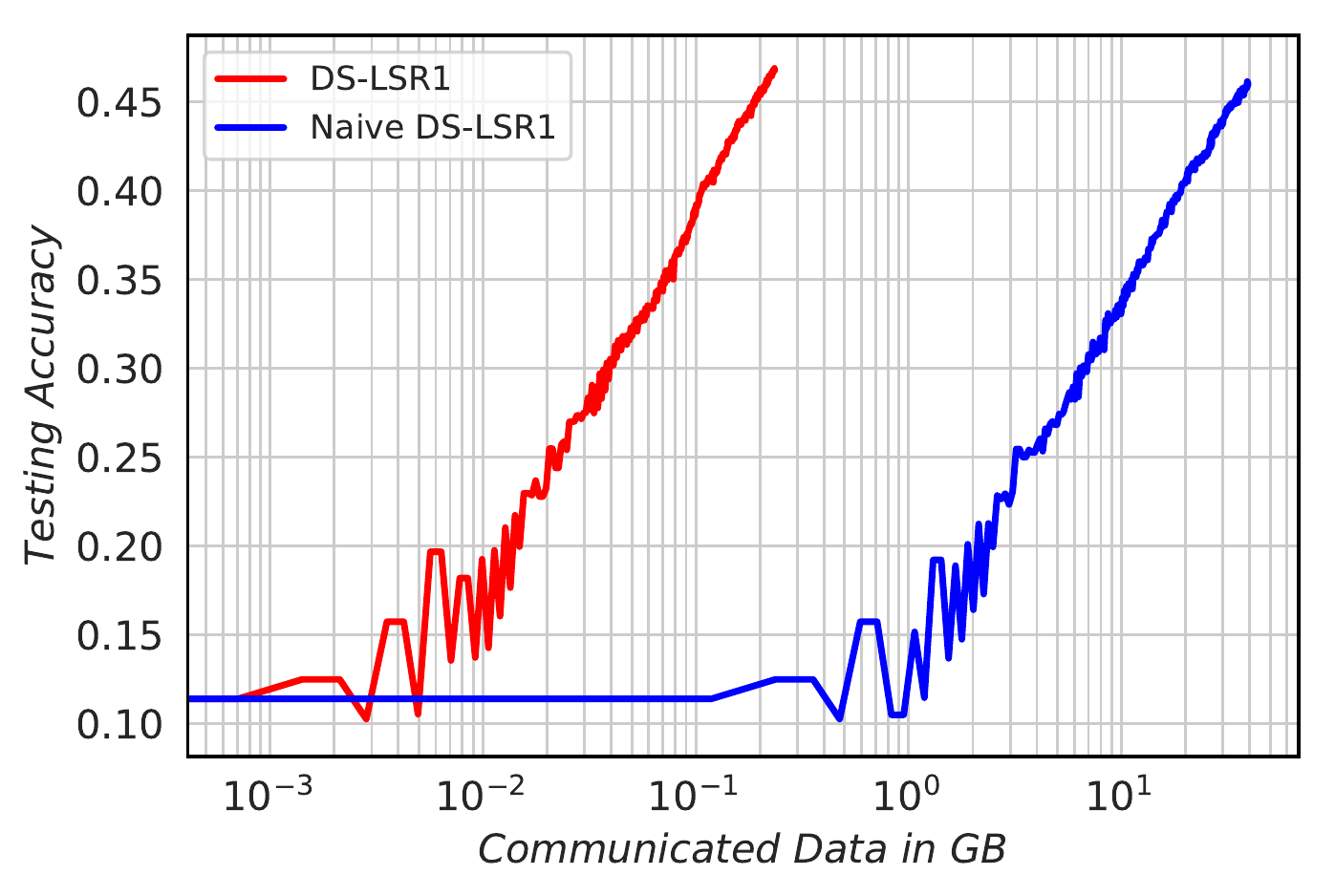}
	
		\caption{Performance of DS-LSR1 on CIFAR10 dataset with different number of nodes.}
		\label{fig:perf_app}
\end{figure}

\end{document}